\documentclass{elsarticle}
\usepackage{amsmath,graphicx}

\usepackage{amsthm}
\usepackage{amssymb}
\usepackage{xypic}
\theoremstyle{plain}
\newtheorem{theorem}{Theorem}
\newtheorem{proposition}[theorem]{Proposition}

\theoremstyle{definition}
\newtheorem{definition}[theorem]{Definition}

\newtheorem{problem}[theorem]{Problem}
\newtheorem{remark}[theorem]{Remark}
\newtheorem{example}[theorem]{Example}

\def\shf{\mathcal}
\def\col{\mathcal}
\def\id{\textrm{id}}
\def\pr{\textrm{pr}}

\title{Sheaves are the canonical data structure for sensor integration}
\author{Michael Robinson\\
Mathematics and Statistics\\
American University\\
Washington, DC, USA\\
michaelr@american.edu}

\begin{document}
\maketitle
\begin{abstract}
A sensor integration framework should be sufficiently general to
accurately represent many sensor modalities, and also be able to
summarize information in a faithful way that emphasizes important,
actionable information.  Few approaches adequately address these two
discordant requirements.  The purpose of this expository paper is to
explain why sheaves are the canonical data structure for sensor
integration and how the mathematics of sheaves satisfies our two
requirements.  We outline some of the powerful inferential tools that
are not available to other representational frameworks.
\end{abstract}
\section{Introduction}

There is increasing concern within the data processing community about
``swimming in sensors and drowning in data,'' \cite{Magnuson_2010} because unifying data across many disparate sources is difficult.  This refrain is repeated
throughout many scientific disciplines, because there are few
treatments of unified models for complex phenomena and it is
difficult to infer these models from heterogeneous data.  

A sensor integration framework should be (1) sufficiently general to accurately
represent all sensors of interest, and also (2) be able to summarize
information in a faithful way that emphasizes important, actionable
information.  Few approaches adequately address these two discordant
requirements.  Models of specific phenomena fail when multiple sensor types are combined into a complex network, because they cannot assemble a global picture consistently.  Bayesian or network theory
tolerate more sensor types, but suffer a severe lack of sophisticated
analytic tools.

The mathematics of \emph{sheaves} partially addresses our two requirements and
provides several powerful inferential tools that are not available to other representational frameworks.  This article presents (1) a sensor-agnostic measure of data consistency, (2) a sensor-agnostic optimization method to fuse data, and (3) sensor-agnostic detection of possible ``systemic blind spots.''  In this article, we show that sheaves provide both theoretical and practical tools to allow representations of locally valid datasets in which the datatypes and contexts vary.  Sheaves therefore provide a common, canonical language for heterogeneous datasets.  We show that sheaf-based fusion methods can combine disparate sensor modalities to dramatically improve target localization accuracy over traditional methods in a series of examples, culminating in Example \ref{eg:sar_data}.  Sheaves provide a sensor-agnostic basis for identifying when information may be inadvertently lost through processing, which we demonstrate computationally in Examples \ref{eg:h1_generator} -- \ref{eg:obstacle_sheaf}.

 Other methods typically aggregate data either exclusively globally (on the level of whole semantic ontologies) or exclusively locally (through various maximum likelihood stages).  This limits the kind of inferences that can be made by these approaches.  For instance, the data association problem in tracking frustrates local approaches (such as those based on optimal state estimation) and remains essentially unsolved in the general case.  The analysis of sheaves avoids both of these extremes by specifying information where it occurs -- locally at each data source -- and then uses global relationships to constrain how these data are interpreted.

The foundational and canonical nature of sheaves means that existing approaches that address aspects of the sensor integration problem space already illuminate some portion of sheaf theory without exploiting its full potential.  In contrast to the generality that is naturally present in sheaves, existing approaches to combining heterogeneous quantitative and qualitative data tend to rely on specific domain knowledge on small collections of data sources.  Even in the most limited setting of pairs quantitative data sources, considerable effort has been expended in developing models of their joint behavior.  Our approach leverages these pairwise models into models of multiple-source interaction.  Additionally, where joint models are not yet available, sheaf theory provides a context for understanding the properties that such a model must have.

\subsection{Contribution}

Sometimes data fusion is described methodologically, as it is in the Joint Directors of Laboratories (JDL) model \cite{Hall_2008,White_1991}, which defines operationally-relevant ``levels'' of data fusion.  A more permissive -- and more theoretically useful -- definition was given by Wald \cite{Wald_1999}, ``data fusion is a formal framework in which are expressed the means and tools for the alliance of data originating from different sources.''  This article shows that this is \emph{precisely} what is provided by sheaf theory, by carefully outlining the requisite background in sheaves alongside detailed examples of collections of sensors and their data.  In Chapter 3 of Wald's book \cite{wald2002data}, the description of the implications of this definition is strikingly similar to what is presented in this article.  But we go further, showing that we obtain not just a full-fidelity representation of the data and its internal relationships, but also a jumping-off point for analysis of both the integrated sensor system and the data it produces.

\begin{figure}
\begin{center}
\includegraphics[width=3in]{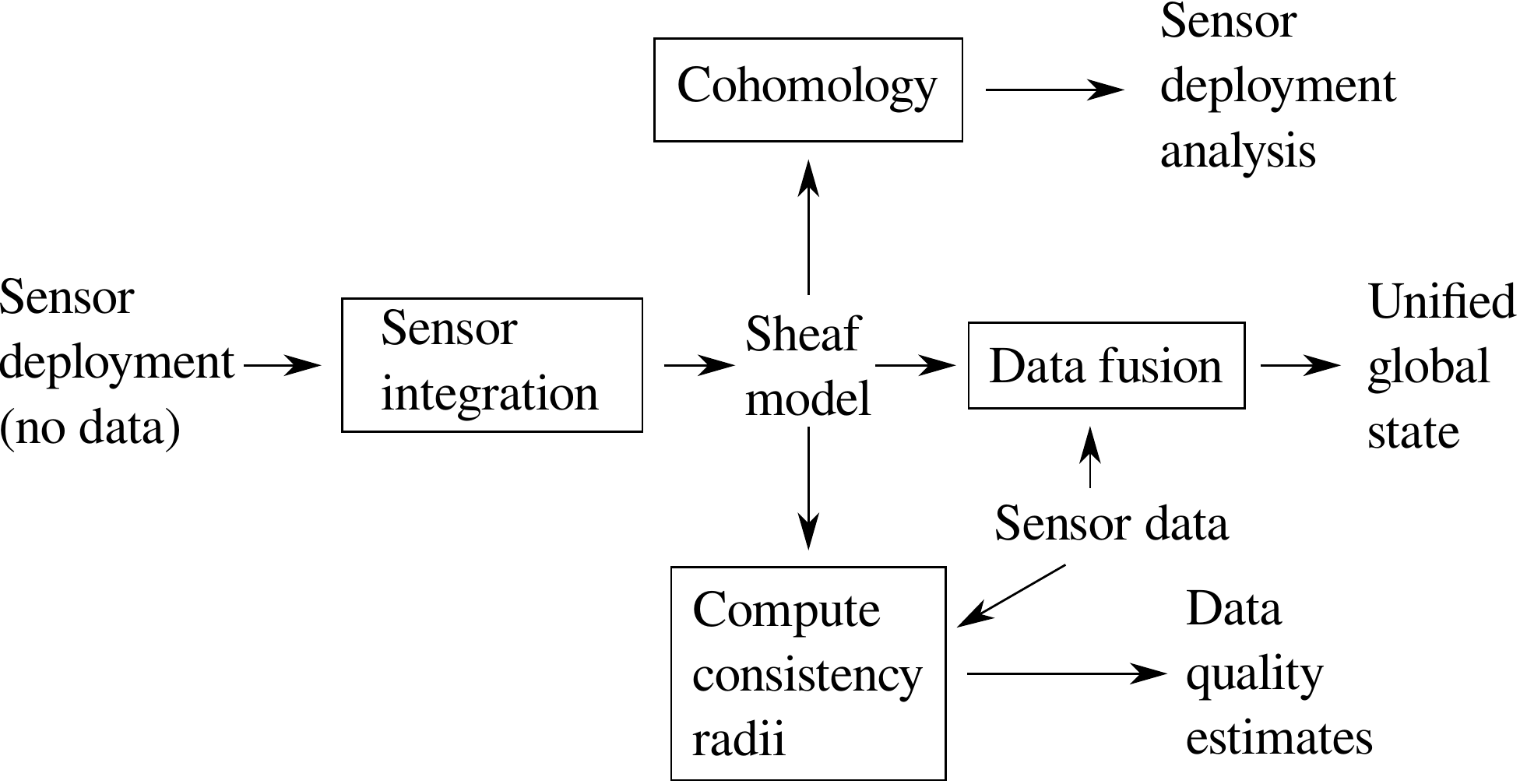}
\caption{The sheaf-based multi-sensor workflow.}
\label{fig:workflow}
\end{center}
\end{figure}

We make the distinction between ``sensor integration'' in which a collection of \emph{models} of sensors are combined into a single unified model (the \emph{integrated sensor system}), and ``sensor data fusion'' in which \emph{observations} from those sensors are combined.  This article presents the sheaf-based workflow outlined by Figure \ref{fig:workflow}, which divides the process of working with sensors and their data into several distinct stages:
\begin{enumerate}
\item \emph{Sensor integration} unifies models of the individual sensors and their inter-relations into a single system model, which we will show has the mathematical structure of a \emph{sheaf}.  We emphasize that no sensor \emph{data} are included in the sensor system model represented by the sheaf. 
\item \emph{Consistency radius computation} uses the sheaf to quantify the level of self-consistency of a set of data supplied by the sensors.
\item \emph{Data fusion} takes the sheaf and data from the sensors to obtain a new dataset that is consistent across the sensor system through a constrained optimization.  The consistency radius of the original data places a lower bound on the amount of distortion incurred by the fusion process.
\item \emph{Cohomology} of the sheaf detects possible problems that could arise within the data fusion process.  Although cohomology does not make use of any sensor data, it quantifies the possible impact that certain aspects of the data may have on the fusion process.
\end{enumerate}

Through a mixture of theory and detailed examples, we will show how this workflow presents solutions to four distinct problem domains:
\begin{enumerate}
\item Formalizing the description of an integrated sensor system as a well-defined mathematical entity (defined by Axioms 1--6 in Section \ref{sec:sheafification}),
\item Quantifying the internal consistency of the data from individual sensors (the \emph{consistency radius} in Definition \ref{def:approx_section_topo} in Section \ref{sec:fusion}),
\item Deriving globally consistent data from the data provided by the entire collection of sensors (solving Problem \ref{prob:fusion}, \emph{sheaf-based data fusion} in Section \ref{sec:fusion}), and
\item Measuring the impact of the relationships between sensors on what data will be deemed consistent (using \emph{cohomology} in Section \ref{sec:cech}).
\end{enumerate}

\subsection{Historical context}

There are essentially two threads of research in the literature on data fusion: (1) physical sensors that return structured, numerical data (``hard'' fusion), and (2) unstructured, semantic data (``soft'' fusion) \cite{khaleghi2013multisensor}.  Hard fusion is generally performed on spatially-referenced data (for example, images), while soft fusion is generally referenced to a common ontology.  Especially in hard fusion, the spatial references are taken to be absolute.  Most \emph{sensor} fusion is performed on a pixel-by-pixel basis amongst sensors of the same physical modality (for instance \cite{varshney1997multisensor,alparone2008multispectral,Zhang_2010}).  It generally requires image registration \cite{Dawn_2010} as a precondition, especially if the sensors have different modalities (for instance \cite{Koetz_2007,Guo_2008}).  When image extents overlap but are not coincident, \emph{mosaics} are the resulting fused product.  These are typically based on pixel- or patch-matching methods (for instance \cite{Chen_2014}).  Because these methods look for areas of close agreement, they are inherently sensitive to differences in physical modality.  It can be difficult to extend these ideas to heterogeneous collections of sensors.

Like hard fusion, soft fusion requires registration amongst different data sources.  However, since there is no physically-apparent ``coordinate system,'' soft fusion must proceed without one.  There are a number of approaches to align disparate ontologies into a common ontology \cite{EuJShP07,EuJFeA09,JoCPaP09a,little2009designing}, against which analysis can proceed.  There, most of the analysis derives from a combination of tools from natural language processing (for instance \cite{hammer1997extracting,riloff1999learning,gregory2011domain}) and artificial intelligence (like the methods discussed in \cite{kushmerick2000wrapper,jakobson2004approach,sarawagi2008information}).  That these approaches derive from theoretical logic, type theory, and relational algebras is indicative of deeper mathematical foundations.  These three topics have roots in the same category theoretic machinery that drives the sheaf theory discussed in this article. 

Weaving the two threads of hard and soft fusion is difficult at best, and is usually approached statistically, as discussed in \cite{waltz1990multisensor,DeGroot_2004,hall2004mathematical}.  Unfortunately, this ``clouds'' the issue.  If a stochastic data source is combined with another data source (deterministic or not), stochastic analysis asserts that the result will be stochastic.  This viewpoint is quite effective in multi-target tracking \cite{smith2006approaches,Newman_2013} or in event detection from social media feeds \cite{alqhtani9multimedia}, where there are sufficient data to estimate probability distributions.  But if two \emph{deterministic} data sources are combined, one numeric and one textual, why should the result be \emph{stochastic}?

Regardless of this conundrum, information theoretic or possibilistic approaches seem to be natural and popular candidates for performing sensor integration, for instance \cite{crowley1993principles,benferhat2006reasoning, benferhat2009fusion}.  They are actually subsumed by category theory \cite{Balduzzi_2011,Thorsen_2006} and arise naturally when needed.  These models tend to rely on the homogeneity of sensors in order to obtain strong theoretical results.  Sheaf theory extends the reach of these methods by explaining that the most robust aspects of networks tend to be topological in nature.  For example, one of the strengths of Bayesian methods is that strong convergence guarantees are available.  However, when applied to data sources arranged in a feedback loop, Bayesian updates can converge to the wrong distribution or not at all! \cite{murphy1999loopy}  The fact that this is possible is witnessed by the presence of a topological feature -- a loop -- in the relationships between sources.

Sheaf theory provides canonical computational tools that sit on the general framework and has been occasionally \cite{Lilius_1993,RobinsonQGTopo,Joslyn_2014} used in applications.  Unfortunately, it is often difficult to get the general machine into a systematic computational tool.  Combinatorial advances \cite{Baclawski_1975,Baclawski_1977,Shepard_1985,Curry} have enabled the approach discussed in this article through the use of finite topologies and related structures.

Sheaves, and their cohomological summaries, are best thought of as being part of \emph{category theory}.  There has been a continuing undercurrent that sheaves aspire to be more than that -- they provide the bridge between categories and logic \cite{Goldblatt}.  The pioneering work by Goguen \cite{goguen1992sheaf} took this connection seriously, and inspired this paper -- our axioms sets are quite similar -- though the focus of this paper is somewhat different.  The approach taken in this paper is that time may be implicit or explicit in the description of a data source; for Goguen, time is always explicit.  More recently category theory has become a valuable tool for data scientists, for instance \cite{Spivak_2014}, and information fusion \cite{Kokar_2004, Kokar_2006}.

\section{Motivating examples}

Although the methodology presented in this article is sufficiently general to handle many configurations of sensors and data types, we will focus on two simple scenarios for expository purposes.  To emphasize that our approach is applicable for many sensor modalities, the two scenarios involve rather different signal models.

\begin{enumerate}
\item A notional search and rescue scenario in which multiple types of sensors are used to locate a downed airliner.
\item A computer vision task involving the tabulation of coins on a table.
\end{enumerate}

We will carry these two scenarios throughout the article.  We will start with the most basic stage of constructing their integrated sensor models in Section \ref{sec:sheafification}, we will then analyze typical sensor data in Section \ref{sec:fusion}, and finally offer a data-independent analysis of the sensor models in Section \ref{sec:cech}.

\subsection{Search and rescue scenario}
The search and rescue scenario involves three distinct kinds of sensors (this is simplified from the guidelines produced by the Australian government \cite{SAR_2014} for the search for Malaysia Airlines Flight MH370) distributed\footnote{The coastline data was obtained from {\tt https://serc.carleton.edu/download/files/79176/coast.mat}} according to Figure \ref{fig:sar_setup}.
\begin{enumerate}
\item The original flight plan, Air Traffic Control (ATC) radars, and the Automatic Dependent Surveillance-Broadcast (ADS-B) aircraft beacon system, all of which produce time-stamped locations of an aircraft,
\item Radio Direction Finding (RDF) sensors provided by amateur radio operators, which produce accurately time-stamped but highly uncertain bearings to an emergency beacon deployed by the aircraft near the crash site, and
\item Wide-area satellite images, which are time-stamped georeferenced intensity-valued images.
\end{enumerate}

\begin{figure}
\begin{center}
\includegraphics[width=4in]{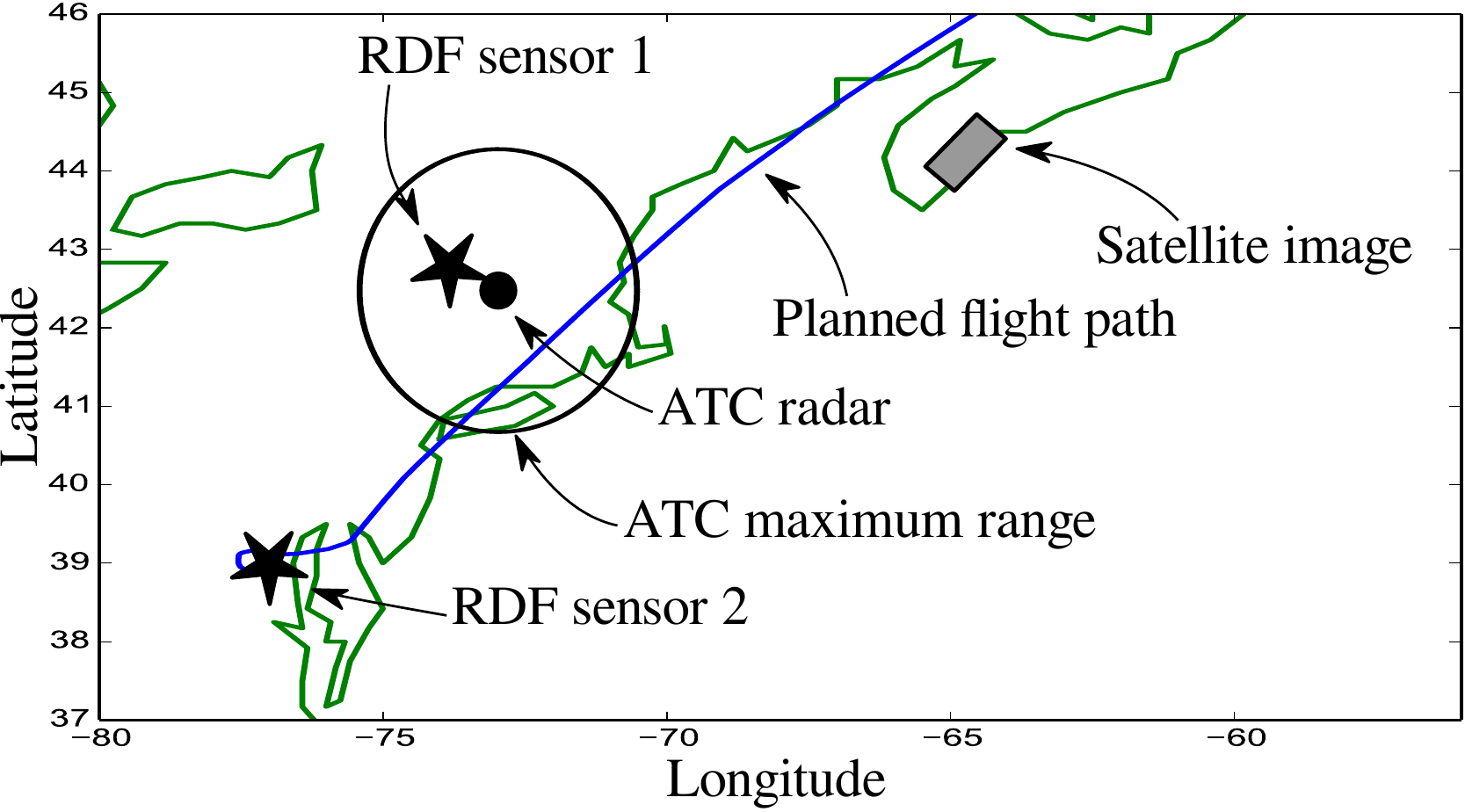}
\caption{The sensors used in the search and rescue scenario for this article.}
\label{fig:sar_setup}
\end{center}
\end{figure}

The data consists of 612 ATC and ADS-B observations.  Of these, 505 observations contained latitude, longitude, altitude and speed, while the other 107 consisted of latitude and longitude only.  The update rate varied considerably over the data, but typical observations were about 10 km apart.  The RDF bearings and satellite images each consisted of exactly one observation each.  Since the data were simulated, various amounts of noise and systematic biases were applied, as detailed in Example \ref{eg:sar_data}. 

The goal of the scenario is to locate the airliner's crash site.  Due to the fact that the most accurate wide-area sensors -- the ATC radars and ADS-B beacons -- do not necessarily report an aircraft's location over the water, this problem is limited by uncertainty about what happened after the airliner leaves controlled airspace.  The time of the crash is (of course!) not part of the flight plan, and was not visible to the ATC radars nor to the satellite.

The time of crash is central to the problem because it connects the time of crash, the last known position, and last known velocity to an estimate of where the crash location was.  When corroborated by the intersection of the RDF bearings, the crash location estimate can be used to task the satellite.  Thus, the problem's solution is to combine the bearings and time-stamps from the RDF sensors with the other sensors, a task performed by the ``field office.''  

\subsection{Computer vision scenario}

In the usual formulation of multi-camera mosaic formation (for instance \cite{varshney1997multisensor,alparone2008multispectral}), several cameras with similar perspectives are focused on a scene.  Assuming that the perspectives and lighting conditions are similar enough, large mosaics can be formed.  The problem is more complicated if the lighting or perspectives differ substantially, but it can still be addressed.  In either case, the methodology aims to align \emph{pixels} or \emph{regions of pixels}.  If some of the images are produced by sensors with different modalities -- say a LIDAR and a RADAR -- pixel alignment is doomed to fail.  Then, the problem is fully heterogeneous, and is usually handled by trying to interpolate into a common coordinate system \cite{Koetz_2007,Guo_2008}.  If there is insufficient sample support or perspectives are too different, then it is necessary to match object-level summaries rather than pixel regions \cite{Zhang_2010}.

\begin{figure}
\begin{center}
\includegraphics[width=5in]{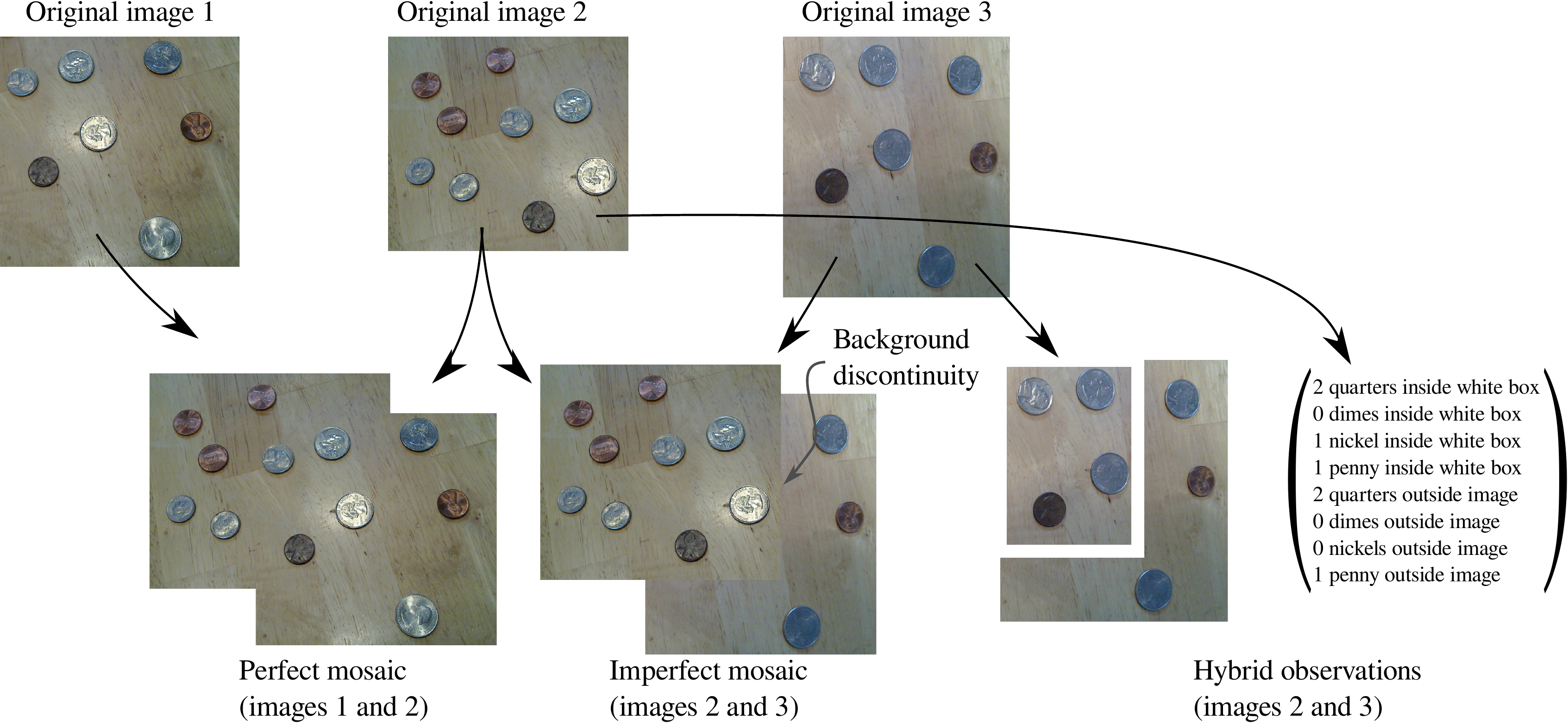}
\caption{Several mosaics (bottom row) of three photographs (top row) under different conditions.  The mosaic on the left is perfect, in that the subimages on the overlap have identical pixel values.  The mosaic in the middle is imperfect, in that the pixel values of the two subimages differ somewhat on the overlap.  The mosaic on the right combines textual and numeric data derived from an object detector run on image 2 with the original image 3.}
\label{fig:mosaics}
\end{center}
\end{figure}

This example studies the task of counting coins on a table.  We will introduce heterogeneity in the imaging modality through different kinds of processing:
\begin{enumerate}
\item All cameras yield images, with identical lighting and perspective: perfect mosaics can be formed (as shown in the lower left mosaic in Figure \ref{fig:mosaics}),
\item All cameras yield images, though lighting and perspective can differ: imperfect mosaics can be formed (as shown in the lower middle mosaic in Figure \ref{fig:mosaics}), and
\item A mixture of camera images and object-level detections derived from the images: a fully heterogeneous data fusion problem (as shown in the lower right mosaic in Figure \ref{fig:mosaics}).
\end{enumerate}

These three examples showcase the transition from homogeneous data in problem (1) to heterogeneous data in problem (3), with homogeneous data being a clear special case.  These examples highlight how fusion of heterogeneous data is to be performed, and makes the resulting structure-preserving summaries easy to interpret.  

\section{Sensor integration}
\label{sec:sheafification}

The first stage of our workflow constructs a unified model of sensors and their interactions.  To emphasize the generality of our approach from the outset -- our methodology is not limited to any particular kind of sensor -- we posit a set of six Axioms to specify what this integrated system must be.  A \emph{sheaf} is precisely any mathematical object satisfying all six Axioms.  We motivate the Axioms by way of examples from our two scenarios, a process that also demonstrates how the Axioms apply for other systems as well.

\begin{description}
\item[Axiom 1:] The entities observed by the sensors lie in a set $X$.
\end{description}

The first Axiom is not controversial.  ``Entities'' might be targets, sensor data, or even textual documents.  

The second Axiom considers properties of the attributes.  In order to quantify differences between attributes, we will use a \emph{pseudometric}.

\begin{definition}
If $A$ is a set, function $d: A \times A \to \mathbb{R}$ is called a \emph{pseudometric} when it satisfies the following:
\begin{enumerate}
\item $d(x,x) = 0$ for all $x\in A$, 
\item $d(x,y) = d(y,x) \ge 0$ for all $x,y \in A$, and
\item $d(x,z) \le d(x,y) + d(y,z)$ for all $x,y,z \in A$.
\end{enumerate}
We call the pair $(A,d)$ a \emph{pseudometric space}.

If additionally, $d(x,y) > 0$ whenever $x \not= y$, we call $d$ a \emph{metric} and $(A,d)$ a \emph{metric space}.
\end{definition}

\begin{description}
\item[Axiom 2:] All possible attributes of the entities lie in a collection\footnote{The examples in this article all make use of a finite collection $\col{A}$, but there is no particular reason why this collection should be finite or even countable.} of pseudometric spaces $\col{A}=\{(A_1,d_1),(A_2,d_2),\dotsc\}$.
\end{description}

We lose no generality by requiring each set of attributes to have a pseudometric, because we may choose the \emph{discrete metric} which asserts that all pairs of distinct attributes are distance 1 apart.  

\begin{example}
\label{eg:sar_axiom1}
Throughout this article, we will use the following entities in the search and rescue scenario:
\begin{itemize}
\item \emph{Aircraft last known position}, represented as its longitude $x$, its latitude $y$, and its altitude $z$,
\item \emph{Aircraft last known velocity}, which is represented as a horizontal velocity vector $(v_x,v_y)$ (we will assume the rate of climb or descent is not known),
\item \emph{Precise time of the crash}, represented as the time elapsed since the last known position $t$,
\item \emph{Bearing to RDF sensor 1}, represented as a compass bearing $\theta_1$ from the crash location to the sensor,
\item \emph{Bearing to RDF sensor 2}, represented as a compass bearing $\theta_2$ from the crash location to the sensor, and
\item \emph{The satellite image}, represented by $s$.
\end{itemize}
Therefore, $X=\{x,y,z,v_x,v_y,t,\theta_1,\theta_2,s\}$ is the set of entities.  The attributes are given by the following sets, each of which comes with a metric:
\begin{itemize}
\item \emph{Three dimensional positions}: $\mathbb{R}^3$ specifying longitude, latitude, and altitude in kilometers under the great circle distance metric,
\item \emph{Two dimensional positions}: $\mathbb{R}^2$ specifying longitude and latitude in kilometers under the great circle distance metric,
\item \emph{Two dimensional velocities}: $\mathbb{R}^2$ specifying longitude and latitude velocities in kilometers/second under the Euclidean metric,
\item \emph{Time}: $\mathbb{R}$ in units compatible with the positions and velocities,
\item \emph{Bearings}: $S^1 = \{(u,v) : u^2 + v^2 = 1\} = [0,360)$ points on the unit circle, represented either by their coordinates or as angles in degrees, and
\item \emph{Images}: $C([0,1]\times[0,1],\mathbb{R})$, continuous real-valued functions on the unit square.  (Satellite images are typically lowpass filtered and oversampled, continuous functions are a reasonable model.  For other imaging modalities, piecewise continuity may be a better model.)
\end{itemize}
Some sensors will also return products of these sets so the set of attributes is all such products $\col{A}=\{\mathbb{R}^3,\mathbb{R}^2,\mathbb{R},S^1,C([0,1]\times[0,1],\mathbb{R}), \mathbb{R}^3\times \mathbb{R}^2, \mathbb{R}^3 \times \dotsc\}$.
\end{example}

Thus far, no attributes have yet been ascribed to entities, so there is nothing to distinguish entities from one another.  Yet since different sets of entities are measured by sensors, it makes sense to study these sets.  

\begin{example}
  \label{eg:sar_sets}
For the search and rescue scenario, several sets of entities will be associated to the sensors as indicated in the list below:
\begin{itemize}
\item \emph{Flight plan}: $\{x,y,z\}$,
\item \emph{ATC radar}: $\{x,y,z,v_x,v_y\}$,
\item \emph{RDF sensor 1}: $\{\theta_1,t\}$,
\item \emph{RDF sensor 2}: $\{\theta_2,t\}$, and
\item \emph{Satellite}: $\{\theta_1,\theta_2,s\}$.
\end{itemize}
The reader is cautioned that each of these are \emph{unordered} sets of entities.  Only when attributes are assigned to each sensor by Axiom 4, will ordering among the entities be apparent (see Example \ref{eg:sar_localsections}).

The field office (where the data fusion is performed) has access to all of the entities.  The sets of entities are permitted to overlap, since the same entity may be observed by multiple sensors.  We usually expect that those pairs of sensors that are providing the same entities will provide similar attributes for these entities.  For instance the flight plan and the ATC should provide similar position information unless there has been a deviation from the plan.  Any such deviation is probably a cause for concern and is something that we want to detect!
\end{example}

One way to formalize the kinds of collections of entities is to construct a \emph{topology} on $X$, which specifies which sets of entities have attributes worthy of comparison.  Intuitively a topology expresses that if two entities are unrelated, they can be described in isolation of each other.  If they are closely related, they will be ``hard to disentangle'' and will usually be described together.  The way to formalize this is by enriching a set model of the universe to describe entity relationships (against which attributes of different entities can then be related) into a topological one.  We recall the definition of a topology:

\begin{definition}
A \emph{topology} on a set $X$ is a collection $\col{T}$ of subsets of $X$ satisfying the following properties:
\begin{enumerate}
\item $X \in \col{T}$.  
\item $\emptyset \in \col{T}$.
\item The union of any collection of elements of $\col{T}$ is also in $\col{T}$.
\item The intersection of any finite collection of elements of $\col{T}$ is also in $\col{T}$.
\end{enumerate}
Usually, the elements of $\col{T}$ are called \emph{open} sets, and the pair $(X,\col{T})$ is called a \emph{topological space}.  In the context of sensors, a $U \in \col{T}$ is called a \emph{sensor domain}.
\end{definition}

The first property of a topology states that it is possible to ascribe attributes to all entities (it may be impractical, though).  The second property states that it is possible to not describe any entities.  The third and fourth properties are more subtle.  They state that two entities can be related through their relationships with other entities.  The finiteness constraint in the fourth property avoids examining an infinite regress through increasingly entity-specific relationships.\footnote{In some topologies, the intersection of any collection of elements remains in the topology, but requiring this would destroy our ability to describe certain phenomena.}

\begin{description}
\item[Axiom 3:] The collection of all sets of entities whose attributes are related forms a topology $\col{T}$ on $X$.
\end{description}

We should emphasize that the sets listed in Example \ref{eg:sar_sets} do \emph{not} form a topology.  There are many intersections, for instance $\{t\} = \{\theta_1,t\} \cap \{\theta_2,t\}$, that do not correspond to a set of entities provided by a sensor.  One must construct the topology by specifying the open sets as those sets of entities provided by the sensors, and then extending the collection until it is a topology.

\begin{definition}
  If $X$ is a set and $\col{U}$ is a collection of subsets of $X$, the smallest topology $\col{T}$ which contains $\col{U}$ is given by
\begin{equation*}
    \col{T} = \bigcap_{\col{T}'\text{ is a topology with }\col{U} \subseteq \col{T}'} \col{T}',
  \end{equation*}
and is called \emph{the topology generated by $\col{U}$}.  In this case, we sometimes say that $\col{U}$ is a \emph{subbase} for $\col{T}$.
\end{definition}

\begin{figure}
\begin{center}
\includegraphics[width=3in]{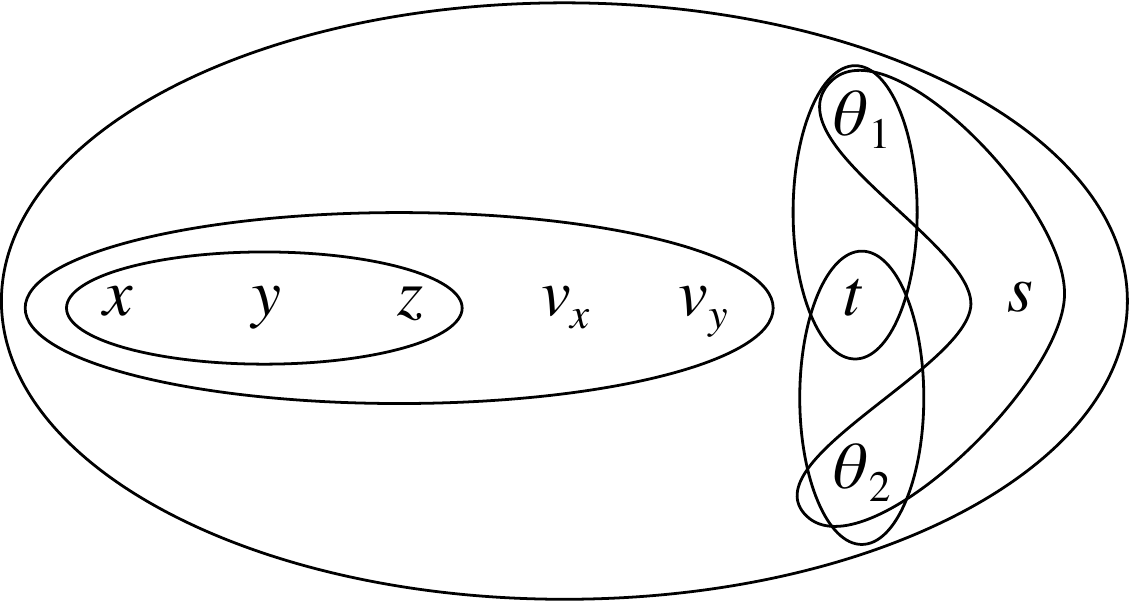}
\caption{Open sets that generate the topology for the search and rescue scenario}
\label{fig:sar_open}
\end{center}
\end{figure}

\begin{example}
  \label{eg:sar_open}
  Figure \ref{fig:sar_open} shows the relationships between the entity sets from Example \ref{eg:sar_sets}.  Using these sets as the subbase for the topology results in many other sets in the topology.  The following open sets are formed as intersections:
  \begin{enumerate}
  \item $\{t\}$,
  \item $\{\theta_1\}$, and
  \item $\{\theta_2\}$.
  \end{enumerate}
  There are many more open sets formed as unions of these sets, for instance
  \begin{enumerate}
  \item $\{\theta_1,\theta_2\}$,
  \item $\{\theta_1,\theta_2,s,t\}$,
  \item $\{x,y,z,\theta_1,\theta_2\}$,
  \item $\{x,y,z,\theta_1,\theta_2,s\}$,
  \item $\{x,y,z,\theta_1,\theta_2,t\}$,
  \item $\{x,y,z,\theta_1,\theta_2,s,t\}$, 
  \end{enumerate}
  among others.  Rather than being dismayed at all the possibilities, observe that many of these open sets represent redundant collections of information about entities.  But which collections should be retained?  For the moment, we will retain all of them.  We will postpone a more complete answer to that question until Section \ref{sec:cech} when we can provide a theoretical guarantee (Theorem \ref{thm:leray}) that information is not lost when only a smaller number of entity sets are used.
\end{example}

Following Example \ref{eg:sar_open}, sensor data will be supplied at some of the open sets in $\col{T}$.  These data are then merged on unions of these open sets whenever the data are consistent along intersections.  The precise relationships among open sets in $\col{T}$ are crucial and generally should have physical meaning.

\begin{remark}
Axiom 3 is a true limitation,\footnote{Instead of a topology, one might consider generalizing further to a Grothendieck \emph{site} \cite{Goldblatt}, but this level of generality seems excessive for practical sensor integration.} which posits that entities are organized into local neighborhoods.  In all situations of which the author is aware, the appropriate topology $\col{T}$ is either the geometric realization of a cell complex or (equivalently) a finite topology built on a preorder as described in \cite{Stong_1966}. The three axioms encompass unlabeled graphs as a special case, for instance.  Finite topologies are particularly compelling from a practical standpoint even though relatively little has been written about them.  The interested reader should consult \cite{May_2003} for a (somewhat dated) overview of the interesting problems.
\end{remark}

\begin{figure}
  \begin{center}
    \includegraphics[width=3in]{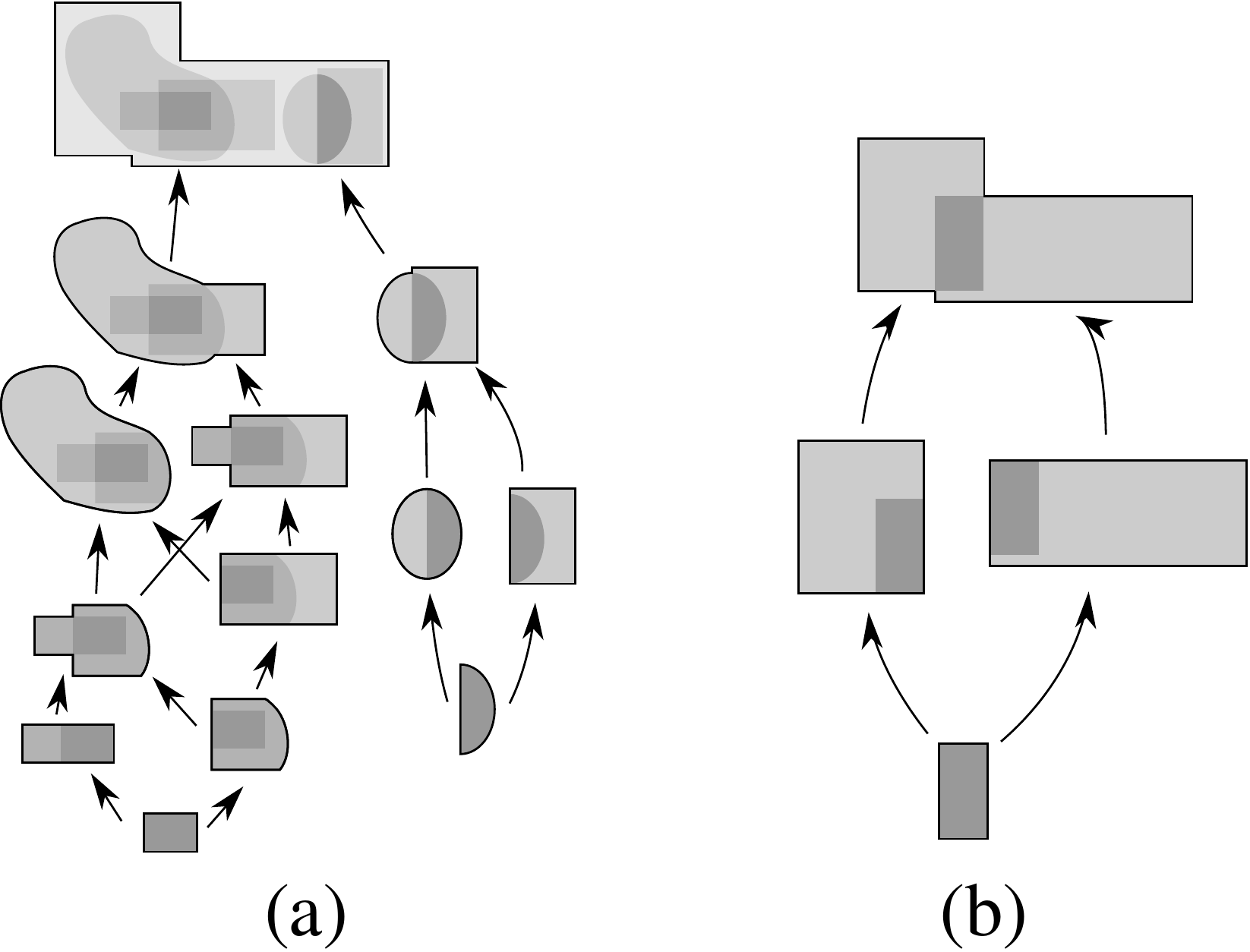}
    \caption{Open sets in two different topologies for image fusion: (a) part of the usual Euclidean topology and (b) a coarser topology generated by two camera fields of view.}
    \label{fig:cam_top}
  \end{center}
\end{figure}

\begin{example}
  \label{eg:cam_top}
In the case of image mosaics, the open sets in $\col{T}$ should have something to do with the regions observed by individual cameras.  Let the set of entities $X$ consist of the locations of points in the scene.  There are two obvious possibilities for a topology:
\begin{enumerate}
\item The topology $\col{T}_1$ of the scene itself (some of which is shown in Figure \ref{fig:cam_top}(a)), generally a manifold or a subspace of Euclidean space based on the pixel values or
\item The topology $\col{T}_2$ which is generated by a much smaller collection of sets $\col{U}$, in which each $U\in\col{U}$ consists of a set of points observed by a single camera (the entirety of this topology is shown in Figure \ref{fig:cam_top}(b)).
\end{enumerate}
Although $\col{T}_1$ is rather natural -- and is what is usually used when forming mosaics -- it has the disadvantage of being quite large.  Figure \ref{fig:cam_top}(a) only shows a small part of one such example!  It also requires fusion to be performed at the pixel level, which is generally inappropriate when there are other heterogeneous sensors that do not return pixel values.  This also can be computationally intense, explains why pixel-level fusion scales badly.  In contrast, $\col{T}_2$ is much smaller (usually finite) and therefore more amenable both to computations and to fusion of heterogeneous data types. 
\end{example}

\begin{description}
\item[Axiom 4:] Each sensor assigns one of the pseudometric spaces of attributes to an open subset of entities in the topological space $(X,\col{T})$.  
\end{description}
\begin{itemize}
\item This assignment is denoted by $\shf{S}: \col{T}\to \col{A}$, so that the attribute pseudometric space assigned to an open set $U\in\col{T}$ is written $\shf{S}(U) \in \col{A}$.
\item The space of attributes $\shf{S}(U)$ is called the space of \emph{observations} over the \emph{sensor domain} $U$.
\item The pseudometric on $\shf{S}(U)$ will be written $d_U$.
\item An element of the Cartesian product $\prod_{U \in \col{T}} \shf{S}(U)$ is called an \emph{assignment\footnote{When all of the sets of observations are vector spaces, \emph{assignments} are usually called \emph{serrations}.} of $\shf{S}$}.
\item We call an element $s \in \shf{S}(X)$ a \emph{global section}.
\end{itemize}

The distinction between a space of attributes and a space of observations is that the latter is directly associated to a particular sensor.  In this article, ``attributes'' may not necessarily be associated to any entity, but ``observations'' are attributes that are assigned to a definite open set of entities.

There are some easy recipes for choosing a space of observations $\shf{S}(U)$ based on the kind of sensor corresponding to the open set $U$:
\begin{itemize}
\item For a camera sensor, the space of observations should be the set of possible images, perhaps denoted as vectors in $\mathbb{R}^{3 m n}$ ($m$ rows, $n$ columns, and $3$ colors for each pixel).  
\item For text documents, the space of observations could be the content of a document as a list of words or a word frequency vector.  
\item If the open set corresponds to the intersection of two sensors of the same type, then usually the observations ought to also be of that type.  For instance, the observations in the overlap between two camera images are smaller images.
\end{itemize}
However, if the open set corresponds to the intersection or union of two different sensors (and therefore not a physical sensor), the set of observations needs to be more carefully chosen and may not correspond to the raw data supplied by any sensor.  This situation is made well-defined by Axioms 5 -- 6.

\begin{example}
  \label{eg:sar_localsections}
  Satisfying Axioms 1-4 in the search and rescue scenario requires pairing the open sets in Example \ref{eg:sar_open} with the attribute sets from Example \ref{eg:sar_axiom1}.  As noted in Example \ref{eg:sar_open}, there are \emph{many} open sets; for brevity's sake, we will not specify them all.  We will denote the assignment of observations $\shf{S}$.  

Axiom 4 merely asks that $\shf{S}$ assign spaces of observations, so for instance the $x$, $y$, and $z$ coordinates could be assigned by $\shf{S}(\{x,y,z\})=\mathbb{R}^3$.  Since the open set $U_1=\{x,y,z\}$ implies no ordering on the entities, it is not yet clear which component of $\mathbb{R}^3$ corresponds to which entity.  Indeed it may be that the correspondence between attributes and entities might be rather complicated.  Given that consideration, those open sets that are associated to sensors according to Example \ref{eg:sar_sets} will be modeled as having spaces of observations as follows:
  \begin{itemize}
  \item \emph{Flight plan}: $\shf{S}(U_1)=\shf{S}(\{x,y,z\})=\mathbb{R}^3$ the last known position according to the flight plan.
  \item \emph{ATC}: $\shf{S}(U_2)=\shf{S}(\{x,y,z,v_x,v_y\})=\mathbb{R}^3\times\mathbb{R}^2$ the last known position and velocity reported by the ATC.
  \item \emph{RDF sensor 1}: $\shf{S}(U_3)=\shf{S}(\{\theta_1,t\})=S^1 \times \mathbb{R}$ bearing and time of beacon's signal as reported by amateur operator 1.
  \item \emph{RDF sensor 2}: $\shf{S}(U_4)=\shf{S}(\{\theta_2,t\})=S^1 \times \mathbb{R}$ bearing and time of beacon's signal as reported by amateur operator 2.
  \item \emph{Satellite}: $\shf{S}(U_5)=\shf{S}(\{s,\theta_1,\theta_2\})=\mathbb{R}^2$ object detection from satellite image.  We could have equally well chosen $\shf{S}(U_5)$ to be the space of raw gray scale images $C([0,1]\times[0,1],\mathbb{R})$ or $\mathbb{R}^{3mn}$, but following standard practice in wide-area search, we assume that the image has already been processed into a list of candidate object detections,\footnote{If an object detection is consistent with a possible crash site, an analyst would be instructed to examine the unprocessed image.} each of which consists of a latitude and longitude only.
  \item \emph{Field office}: $\shf{S}(X)=\shf{S}(\{x, y, z, t, v_x, v_y, \theta_1, \theta_2, s\}) = \mathbb{R}^3\times\mathbb{R}^2\times\mathbb{R}$ last known position, last known velocity, time of beacon, and satellite image.  We note that because the satellite image has been processed into object detections, the unprocessed image is not represented in the attribute space.
  \end{itemize}
  For those open sets associated to intersections, the spaces of observations are easy to specify by reading off the entities in the open sets:
  \begin{enumerate}
  \item $\shf{S}(\{t\}) = \mathbb{R}$, the time of the crash,
  \item $\shf{S}(\{\theta_1\}) = S^1$, the bearing of the beacon to amateur operator 1, and
  \item $\shf{S}(\{\theta_2\}) = S^1$, the bearing of the beacon to amateur operator 2.
  \end{enumerate}
  But what about sets like $\{\theta_1,\theta_2\}$? There are two obvious possibilities: $S^1 \times S^1$ (two bearings) versus $\mathbb{R}^2$ (the intersection of the bearing sight lines).  Axiom 4 is mute about which is preferable, but it asserts that we must choose \emph{some} set.  Each open set formed by unions of the subbase is a place where such a decision must occur.  Since Axioms 5 and 6 completely constrain these spaces of observations, we will delay their specification until after these Axioms are specified in Example \ref{eg:sar_full_def}.
\end{example}

\begin{example}
  \label{eg:camera_radar}
On an open set $U\cap V$ corresponding to the intersection between a camera image $U$ and a radar image $V$, each observation in $\shf{S}(U\cap V)$ could be a list of target locations rather than the data from either sensor.  If there are $n$ targets with known locations in $\mathbb{R}^3$, then clearly this space is parameterized by $\mathbb{R}^{3n}$.  If the number of targets is unknown, but is known to be less than $n$, then the space of observations would be given by
\begin{equation*}
  \{\perp\}\sqcup \mathbb{R}^3 \sqcup \mathbb{R}^6 \sqcup \dotsb \mathbb{R}^{3n}
\end{equation*}
where $\sqcup$ represents disjoint union and $\perp$ represents the state of having no targets in the scene, and hence represents a null observation. 
\end{example}

This example indicates that moving from the observations over one open set to a smaller open set ought to come from a particular transformation that depends on the sensors involved.  Put another way, removing entities of a sensor domain from consideration should impact the set of observations in a systematic way.

\begin{description}
\item[Axiom 5:] It is possible to transform sets of observations to new sets of observations by reducing the size of sensor domains.
\end{description}
\begin{itemize}
\item Whenever a pair of sensors assign attributes to open sets of entities $U \subseteq V \subseteq X$ in the topology $\col{T}$, there is a (fixed) continuous function\footnote{Notationally, $\shf{S}(U)$ is a pseudometric space, but $\shf{S}(U\subseteq V)$ is a function between pseudometric spaces.  $\shf{S}$ itself refers to \emph{both} as a single mathematical object.}
  \begin{equation*}
    \shf{S}(U\subseteq V):\shf{S}(V)\to \shf{S}(U)
  \end{equation*}
  from the pseudometric space of observations over the larger sensor domain to the pseudometric space of observations over the smaller one.
\item If $U \subseteq V \subseteq W$ and $U \subseteq Y \subseteq W$, then
  \begin{equation*}
    \shf{S}(U \subseteq V)\circ\shf{S}(V \subseteq W) = \shf{S}(U \subseteq Y)\circ\shf{S}(Y \subseteq W),
  \end{equation*}
  which means that we may refer to $\shf{S}(U \subseteq W)$ without ambiguity.
\item $\shf{S}(U \subseteq U)$ is the identity function.
\item The continuous function $\shf{S}(U \subseteq V)$ is called a \emph{restriction}, because it restricts observations on an open set $V$ to a subset $U$.
\item Any $\shf{S}$ that satisfies Axioms 1--5 is called a \emph{presheaf of pseudometric spaces}.  Analogously, if $d_U$ is a metric for all $U \in \col{T}$, we say $\shf{S}$ is a \emph{presheaf of metric spaces}. 
\end{itemize}

The topology $\col{T}$ imposes a lower bound on how small sensor domains can be, and therefore how finely restrictions are able to discriminate between individual entities.  Generally speaking, one wants the topology to be as small (coarse) as practical, so that fewer restrictions need to be defined (and later checked) during the process of fusing data.

\begin{figure}
\begin{center}
\includegraphics[width=5in]{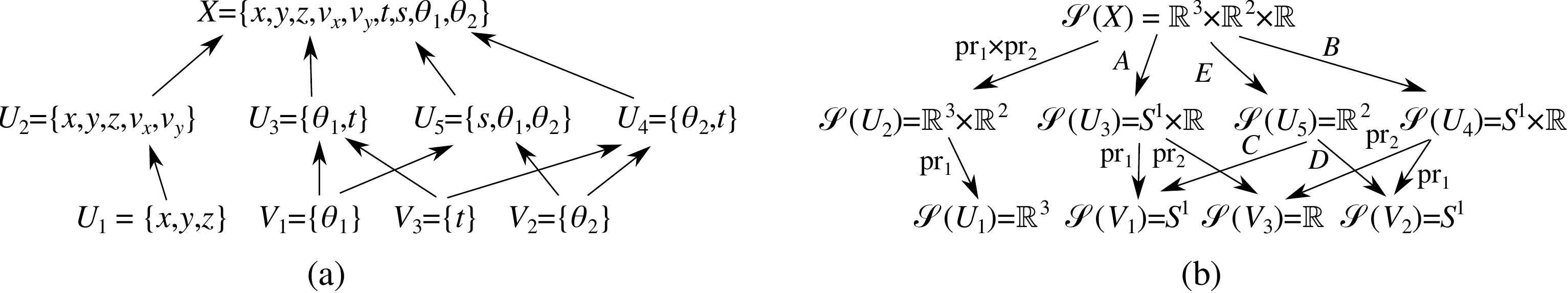} 
\caption{Some of restrictions for the search and rescue scenario; to reduce clutter, unions between open sets are not shown.  (a) The open sets and their intersections, where the arrows indicate subset relations. (b) The spaces of sections, with restrictions shown as arrows.  The restrictions $A$, $B$, $C$, $D$, and $E$ are defined in Example \ref{eg:sar_rest}. }
\label{fig:sar_rest}
\end{center}
\end{figure}

\begin{example}
  \label{eg:sar_rest}
  Continuing with the construction of $\shf{S}$ from Example \ref{eg:sar_localsections} for the search and rescue scenario, restrictions can be constructed as shown in Figure \ref{fig:sar_rest}, where $\id$ represents an identity function and $\pr_k$ represents a projection onto the $k$-th factor of a product.  Except for the restrictions $A$, $B$, $C$, $D$, and $E$, the restrictions merely extract the spaces of observations for various entities that are retained when changing open sets.  For instance, when relating the entities as viewed by the ATC to the flight plan, one merely leaves out the velocity measurements.

  The restrictions $A$, $B$, $C$, $D$, and $E$ need to perform some geometric computations to translate between compass bearings and locations.  They are defined\footnote{We are using $\tan^{-1}$ for brevity -- really something like the C language {\tt atan2} function would be preferable to obtain angles anywhere on the circle.} as follows:
  \begin{eqnarray*}
    A(x,y,z,v_x,v_y,t) &=& \left(\tan^{-1} \frac{x+v_xt -r_{1x}}{y+v_yt -r_{1y}},t\right),\\
    B(x,y,z,v_x,v_y,t) &=& \left(\tan^{-1} \frac{x+v_xt -r_{2x}}{y+v_yt - r_{2y}},t\right),\\
    C(s_x,s_y) &=& \tan^{-1} \frac{s_x - r_{1x}}{s_y - r_{1y}},\\
    D(s_x,s_y) &=& \tan^{-1} \frac{s_x - r_{2x}}{s_y - r_{2y}}, \text{ and} \\
    E(x,y,z,v_x,v_y,t) &=& (x+v_x t, y+v_y t)
  \end{eqnarray*}
  where $s_x,s_y$ are coordinates of an object detected in the satellite image, $r_{1x},r_{1y}$ are the coordinates of the first RDF sensor and $r_{2x},r_{2y}$ are the coordinates of the second RDF sensor.
  Note that restrictions involving unions of open sets are not shown, since we are still waiting to axiomatically constrain them using Axiom 6 in Example \ref{eg:sar_full_def}.
 \end{example}

\begin{example}
  Axiom 5 has quite a clear interpretation for collections of cameras -- the restrictions perform cropping of the domains.  The topology $\col{T}$ controls how much cropping is permitted.  Figure \ref{fig:crop} gives two examples, in which the topology on the left is smaller than the one on the right.  Indeed, if we consider the topology $\col{T}_1$ from Example \ref{eg:cam_top}, then a huge number of crops must be defined, including many that are unnecessary to form a mosaic of the images.  However, the topology $\col{T}_2$ from that example requires the modeler to define only those crops necessary to assemble a mosaic from the images.
\end{example}

\begin{figure}
\begin{center}
\includegraphics[width=4in]{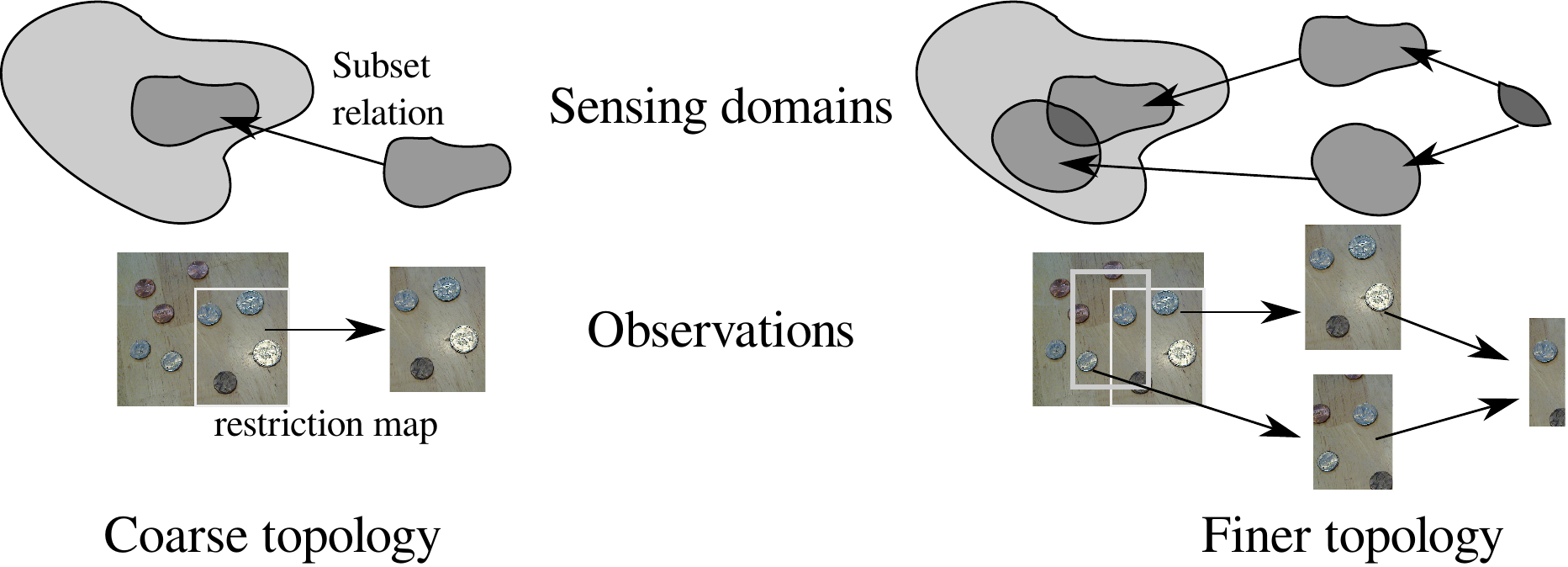}
\caption{Cropping an image corresponds to reducing the size of the sensor domain.  The topology determines what crops can be performed as restriction maps: a coarse topology (left) permits fewer crops than a finer one (right).}
\label{fig:crop}
\end{center}
\end{figure}
 
Without constraints on the nature of the restriction functions, it is in general impossible to combine observations from overlapping sensor domains because we lack a concrete definition for the sets of observations over unions.  Axiom 6 supplies a necessary and sufficient constraint, by requiring there to be an assignment of observations over a union of sensor domains whenever there are observations that agree on intersections of sensor domains.  Axiom 6 asserts that this ``globalized'' observation (a \emph{section} of $\shf{S}$) is unique -- a statement of the underlying determinism of the representation.

\begin{description}
\item[Axiom 6:] Observations from overlapping sensor domains that are equal upon restricting to the intersection of the sensor domains uniquely determine an observation from the union of the sensor domains.
\end{description}
\begin{itemize}
\item Suppose that $\{U_\alpha\}_{\alpha\in I}$ is a collection of sensor domains and there are corresponding observations $a_\alpha \in \shf{S}(U_\alpha)$ indexed by $\alpha \in I$ taken from each sensor domain.  If the following equation\footnote{This equation is called \emph{conjunctivity} in the sheaf literature.} is satisfied
\begin{equation*}
\left(\shf{S}(U_\alpha\cap U_\beta \subseteq U_\alpha)\right)(a_\alpha) = \left(\shf{S}(U_\alpha\cap U_\beta \subseteq U_\beta)\right)(a_\beta),
\end{equation*}
for all $\alpha$ and $\beta$ in $I$, then there is an $a \in \shf{S}(\cup_{\gamma \in I} U_\gamma)$ that restricts to each $a_\alpha$ via 
\begin{equation*}
a_\alpha = \left(\shf{S}(U_\alpha \subseteq \cup U_\gamma)\right)(a).
\end{equation*}
\item Suppose that $U$ is a sensor domain and that $a,b \in \shf{S}(U)$ are assignments of observations to that sensor domain.  If these two observations are equal on all sensor domains $V$ in a collection $\{V_\alpha\}$ for $\alpha \in I$ of sensor domains with $\bigcup_{\alpha \in I} V_\alpha = U$, 
\begin{equation*}
\shf{S}(V \subseteq U)(a) = \shf{S}(V \subseteq U)(b),
\end{equation*}
then we can conclude\footnote{A presheaf satisfying this equation is called a \emph{monopresheaf}.} that $a=b$.
\item When $\shf{S}$ satisfies Axioms 1--6, we call $\shf{S}$ a \emph{sheaf of pseudometric spaces}.
\end{itemize}

In the case of two sensor domains $U_1, U_2,$ Axiom 6 can be visualized as a pair of \emph{commutative diagrams}
\begin{equation*}
  \xymatrix{
    & U_1 \cup U_2 & & & \shf{S}(U_1 \cup U_2) \ar[dl]_{\shf{S}(U_1 \subseteq U_1 \cup U_2)}\ar[dr]^{\shf{S}(U_2 \subseteq U_1 \cup U_2)} & \\
    U_1 \ar[ur]^{U_1 \subseteq U_1\cup U_2} & & U_2 \ar[ul]_{U_2 \subseteq U_1\cup U_2} & \shf{S}(U_1) \ar[dr]_{\shf{S}(U_1\cap U_2 \subseteq U_1)} & & \shf{S}(U_2) \ar[dl]^{\shf{S}(U_1\cap U_2 \subseteq U_2)} & \\
    & U_1 \cap U_2 \ar[ur]_{U_1 \cap U_2 \subseteq U_2}\ar[ul]^{U_1\cap U_2 \subseteq U_1} & & & \shf{S}(U_1 \cap U_2) & \\
    }
\end{equation*}
in which the arrows on the left represent subset relations, while the arrows on the right represent restriction maps.  The Axiom asserts that pairs of elements on the middle row that restrict to the same element on the bottom row are themselves restrictions of an element in the top row.  Therefore Axiom 6 can always be satisfied if $\shf{S}(U_1 \cup U_2)$ has not been constrained by the sensor models.  One merely needs to define
\begin{equation}
\label{eq:pullback} 
\shf{S}(U_1 \cup U_2) = \{(x,y) \in \shf{S}(U_1) \times \shf{S}(U_2) : \left(\shf{S}(U_1 \cap U_2 \subseteq U_1)\right)(x) =  \left(\shf{S}(U_1 \cap U_2 \subseteq U_2)\right)(y)\}
\end{equation}
and let the restriction maps project out of the product in the diagram above.  

\begin{remark}
Axiom 6 implies that the sets of observations cannot get larger as the open sets get smaller.  For instance, the diagram 
\begin{equation*}
\xymatrix{
&\shf{S}(U_1\cup U_2)=\{a\}\ar[dl]_{f}\ar[dr]_{\id}&\\
\shf{S}(U_1)=\{a,c\}\ar[dr]_g&&\shf{S}(U_2)=\ar[dl]_{\id}\{a\}\\
&\shf{S}(U_1 \cap U_2)=\{a\}&
}
\end{equation*}
in which $f(a)=a$, $g(a)=g(c)=a$, and $\id$ is the identity function violates Axiom 6.  Specifically, Axiom 6 would imply that given that $c \in \shf{S}(U_1)$ and $a\in\shf{S}(U_2)$ both map to $a\in \shf{S}(U_1 \cap U_2)$, there should be an element of $\shf{S}(U_1 \cup U_2)$ mapping to $c \in \shf{S}(U_1)$ even though this simply does not occur.
\end{remark}

\begin{remark}
  If the observations consist of probability distributions, it is usually best to think of ``equality'' in Axiom 6 as being ``equality almost surely'' throughout the network of sensors, rather than pointwise equality.
\end{remark}

\begin{example}
  \label{eg:sar_full_def}
  It is now possible to complete the definition of the sheaf model for the search and rescue scenario, picking up where we left off in Example \ref{eg:sar_rest}.  While enumerating all of the spaces of observations over possible unions in the topology would be tedious, a few are enlightening.

\begin{itemize}
\item Consider first the case of the union of $U_1$ (flight plan) and $U_2$ (ATC radar), for which we have already defined
\begin{eqnarray*}
\shf{S}(U_1) &=& \mathbb{R}^3,\\
\shf{S}(U_2) &=& \mathbb{R}^3\times\mathbb{R}^2, \text{ and }\\
U_1\cap U_2 &=& U_1.
\end{eqnarray*}
Therefore, Equation \eqref{eq:pullback} defines 
\begin{equation*}
\shf{S}(U_1 \cup U_2) = \shf{S}(U_2) = \mathbb{R}^3 \times \mathbb{R}^2.
\end{equation*}

\item For the union of $U_1$ (flight plan) and $U_3$ (RDF sensor 1), we have that $U_1 \cap U_3 = \emptyset$, so Equation \eqref{eq:pullback} defines
\begin{equation*}
\shf{S}(U_1 \cup U_3) = \shf{S}(U_1) \times \shf{S}(U_3) = \mathbb{R}^3 \times S^1 \times \mathbb{R}.
\end{equation*}

\item Finally, the open sets $U_3$, $U_4$ corresponding to the two RDF sensors have $U_3 \cap U_4 = \{t\}$.  In Example \ref{eg:sar_localsections}, we defined $\shf{S}(\{t\}) = \mathbb{R}$ to represent the time of the crash measured by both beacons.  Thus Equation \eqref{eq:pullback} constructs
\begin{equation*}
\shf{S}(U_3 \cup U_4) = S^1 \times \mathbb{R} \times S^1,
\end{equation*}
which represents both bearings and the common time from the two RDF sensors.  Global sections over $U_3 \cup U_4$ explicitly assert that the two sensors are in agreement about the time the beacon was sounded.  If the two sensors do not agree, this cannot correspond to a global section!
\end{itemize}
\end{example}

\begin{figure}
\begin{center}
\includegraphics[width=5in]{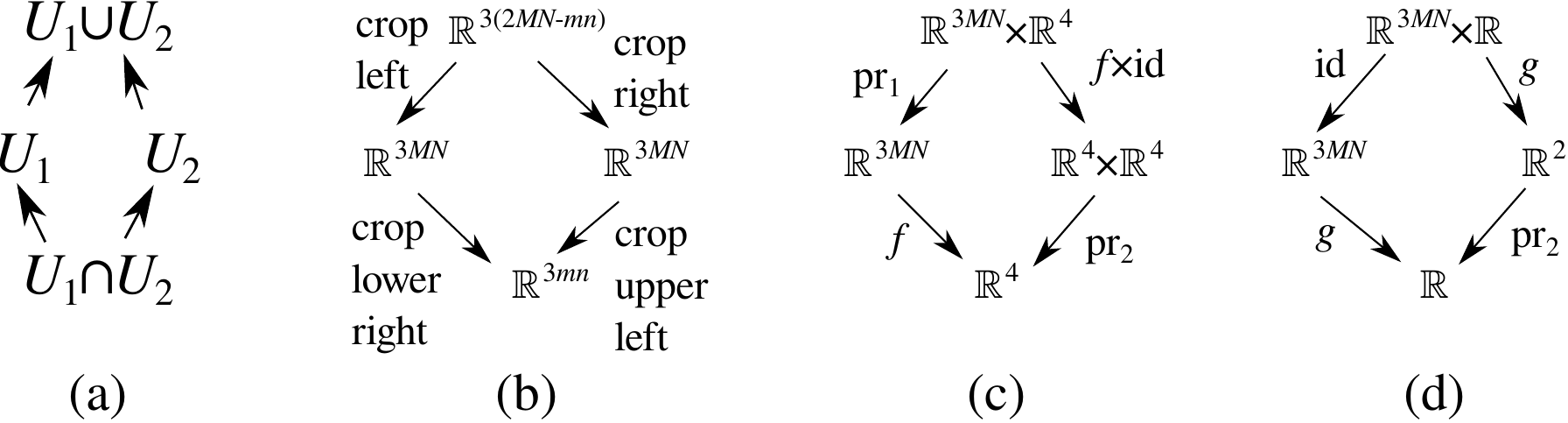}
\caption{Sheaf models for three ways to combine information about coins on a table.  (a) The topology for two overlapping camera fields of view.  (b) Mosaicing the images directly. (c) processing the right image with object-level classifications.  (d) processing the right image into the total value of the coins.}
\label{fig:coins_presheaf}
\end{center}
\end{figure}

\begin{figure}
\begin{center}
\includegraphics[width=5in]{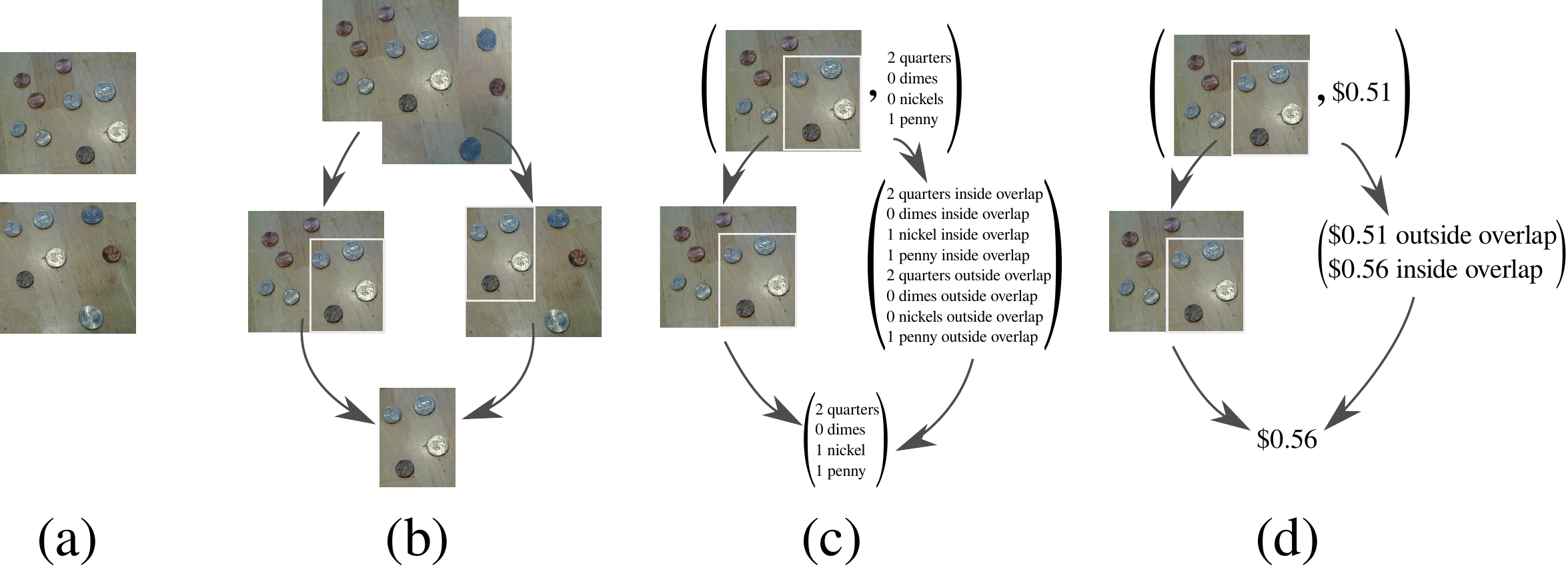}
\caption{Three possible ways to combine information about coins on a table.  (a) The original images. (b) forming the mosaic of two images, with no further interpretation.  (c) fusing an image with object-level classifications, but no interpretation.  (d) combining an image with the total value of the coins, which requires interpretation of objects in the scene.}
\label{fig:coins}
\end{center}
\end{figure}

With Axiom 6, it becomes possible to draw global inferences from data specified locally.
\begin{example}
  \label{eg:coins}
  Referring to the topology $\col{T}_2$ shown Figure \ref{fig:coins_presheaf}(a) (also in Figure \ref{fig:cam_top}(b)), we again consider the problem of counting coins on the table.  Assume that there are three different types of coins: pennies, nickels, dimes, and quarters.  Three different possibilities for processing chains lead to three distinct sheaves, shown in Figure \ref{fig:coins_presheaf}:
  \begin{enumerate}
  \item No processing aside from image registration (Figure \ref{fig:coins_presheaf}(b)),
  \item Leaving one image (on the left) alone while running an object classifier on the other (Figure \ref{fig:coins_presheaf}(c)) in which the restriction $f$ computes the number of coins of each type present in the image, or
  \item Leaving one image (again on the left) alone, while running a different process $g$ that computes the total monetary value within an image (Figure \ref{fig:coins_presheaf}(d)).
  \end{enumerate}
  Notice that in the latter two cases, the observation in the intersection is \emph{not} an image, but something rather more processed.  Given that $f$ and $g$ actually perform the processing required of them, it is easy to verify that Axiom 6 will hold for each sheaf merely by verifying that Equation \ref{eq:pullback} is satisfied over the unions.  

Starting with the original pair of images in Figure \ref{fig:coins}(a), which constitute observations over $U_1$ and $U_2$ in the sheaf in Figure \ref{fig:coins_presheaf}(b), Figure \ref{fig:coins}(b) shows the corresponding global section.  In the case of the sheaves shown in Figures \ref{fig:coins_presheaf}(c) and \ref{fig:coins_presheaf}(d) we can still use the image as an observation over $U_1$.  However, the image that was an observation over $U_2$ must be processed into object detections or a total dollar amount before it can constitute an observation for the sheaves in Figures \ref{fig:coins_presheaf}(c) and \ref{fig:coins_presheaf}(d).  Once that is done, then Figures \ref{fig:coins}(c) and \ref{fig:coins}(d) constitute the resulting global sections that are guaranteed to exist by Axiom 6.
\end{example}

The previous examples show that observations over some of the sensor domains can be used to infer observations over all other sensor domains when there is self consistency among the observations.  This extrapolation of observations into other sensor domains is the hallmark of a data fusion process, but depends on restrictions of observations being exactly equal on intersections of sensor domains.  Of course, this is quite unrealistic!  In Section \ref{sec:fusion} we argue that \emph{data fusion} consists of the process of identifying the global section that is the ``best match'' to the available data.

Sometimes, as the next example shows, problems arise when the integrated sensor system does not satisfy Axiom 6.

\begin{figure}
\begin{center}
\includegraphics[width=4in]{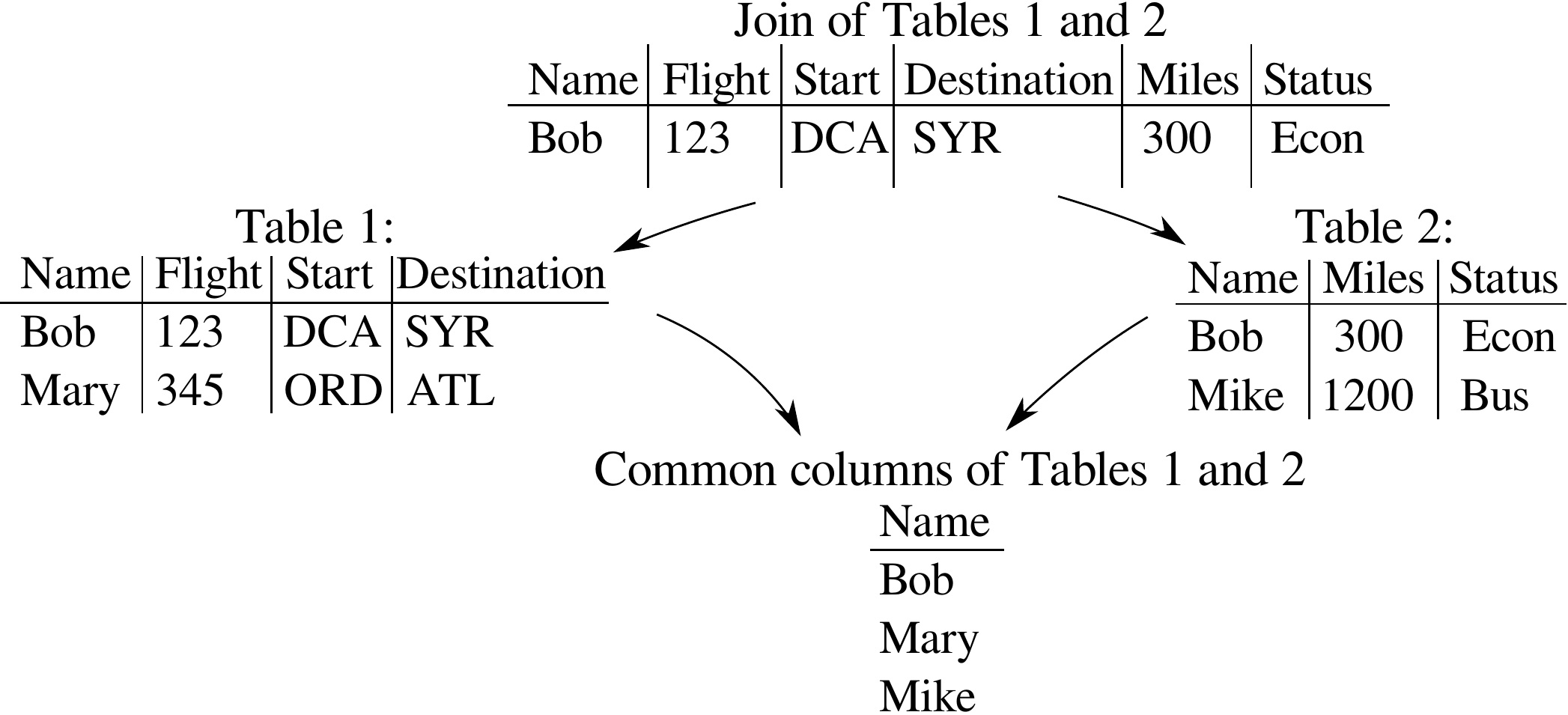}
\caption{The diagram of the sheaf modeling two tables in a passenger database owned by an airline.}
\label{fig:airline}
\end{center}
\end{figure}

\begin{example}
\label{eg:database}
  Real world examples of relational databases often do not satisfy Axiom 6, which is sometimes a source of problems.  Consider a database with two tables like so:
  \begin{enumerate}
  \item a table whose columns are passenger names, flight numbers, a starting airport, and a destination airport, and
  \item a table whose columns are passenger names, frequent flier miles, and flier upgrade status.
  \end{enumerate}
  The two tables are shown at left and right middle in Figure \ref{fig:airline}.  The restrictions (shown by arrows in the Figure) correspond to projecting out the columns to form subtables.  Axiom 6 ensures that if the two tables each have a row with the same passenger's name, then there really is at least one passenger with all of the associated attributes (start, destination, miles, and status).  This is a table join operation: rows from the two tables that match correspond to rows in the joined table.  Axiom 6 also ensures that there is \emph{exactly one} such passenger, so in Figure \ref{fig:airline} this means that ``Bob'' in Table 1 really is the same as ``Bob'' in Table 2.  From personal experience -- the author has a relatively common name -- Axiom 6 does not hold for some airlines. 
\end{example}

\section{Data fusion and quality assessment}
\label{sec:fusion}

The previous section discussed how the relationships between sensors and their possible spaces of observations fit together into an integrated sensor system.  Using the formal model of a sensor system as a presheaf $\shf{S}$ (satisfying Axioms 1--5), it becomes possible to talk about combining the data from sensors into a more unified perspective.  The idea is that if two sensors describe overlapping sets of entities, their observations can be combined when they ``agree'' on the common entities.  Since realistic data contains errors and uncertainties, it is rarely the case that ``agreement'' will be exact.  This suggests formalizing the distance between assignments, which has the structure of a pseudometric space.

\begin{definition}
  \label{df:assignment_pmetric}
  The space of assignments $\prod_{U \in \col{T}} \shf{S}(U)$ has a pseudometric, which is given by
  \begin{equation*}
    D(a,b) = \sup_{U \in \col{T}} d_U(a(U),b(U))
  \end{equation*}
  for two assignments $a,b$ of $\shf{S}$.
\end{definition}

Two assignments are close to one another whenever they consist of observations that are close in all sensor domains.  In Example \ref{eg:sar_localsections}, assignments consist of one report from each sensor, each intersection, and from the field office.  We should expect that the reports ought to be related -- especially between a pair of sensors and their \emph{joint} report on the same entities -- so if one such report is awry, we should suspect that something may be amiss.  Therefore, the central problem we must address is an optimization.
\begin{problem} (Data Fusion)
  \label{prob:fusion}
  Given an assignment $a \in \prod_{U \in \col{T}} \shf{S}(U)$, find the nearest global section $s$ of $\shf{S}$, namely construct
  \begin{eqnarray*}
    \text{argmin}_{s \in \shf{S}(X)} D(s,a)&=&\text{argmin}_{s \in \shf{S}(X)} \sup_{U \in \col{T}} d_U(a(U),s(U))\\
 &=& \text{argmin}_{s \in \shf{S}(X)} \sup_{U \in \col{T}} d_U(a(U),\left(\shf{S}(U \subseteq X)\right)(s)).
  \end{eqnarray*}
\end{problem}

The constraint $s\in \shf{S}(X)$ on the argmin means that the optimization must proceed over all global sections of the presheaf $\shf{S}$.  Thus, the presheaf structure will usually play a substantial role in constraining the solutions of Problem \ref{prob:fusion} through its space of global sections.  As a practical matter, this means that the space of global sections must be determined \emph{before} one can perform the optimization.  The global sections are easily found by inspection in most of the examples in this article, as demonstrated in Examples \ref{eg:sar_full_def} -- \ref{eg:database}.  For more complicated situations, provided one has a sheaf of vector spaces (and not just a presheaf), the space of global sections can be computed via linear algebra (see Proposition \ref{prop:cohomology_global} in Section \ref{sec:cech}).  Without the presheaf structure, it is impossible to even pose Problem \ref{prob:fusion}!

If each set of observations lies in a vector space, then Data Fusion is an optimization over vector spaces, which a well-studied and well-understood setting.  Our contribution is an appropriate pseudometric, which will depend intimately on the presheaf $\shf{S}$.  If any set of observations is not a vector space, then Data Fusion will need to be solved by some other optimization method.  Regardless, solving Data Fusion on a coarser topology is easier than on a finer one, since coarser topologies have fewer open sets.

In order to understand conditions under which this Problem may be solved, we make the following definition to quantify the self-consistency of an assignment.

\begin{definition}
  \label{def:approx_section_topo}
An \emph{$\epsilon$-approximate section} for a presheaf of pseudometric spaces $\shf{S}$ is an assignment $s\in\prod_{U \in \col{T}} \shf{S}(U)$ for which 
\begin{equation*}
d_V(s(V),\shf{S}(V \subseteq U)s(U)) \le \epsilon
\end{equation*}
for all $V \subseteq U$.  The minimum value of $\epsilon$ for which an assignment $s\in\prod_{U \in \col{T}} \shf{S}(U)$ is an $\epsilon$-approximate section is called the \emph{consistency radius} of $s$.
\end{definition}

Smaller consistency radii imply better agreement amongst sensors.  Theoretically, the best one can do is exact equality, in which approximate sections closely correspond to global sections.

\begin{proposition}
\label{prop:global_to_approx}
If $\shf{S}$ is a presheaf of pseudometric spaces, every global section $s\in \shf{S}(X)$ induces a 0-approximate section of $\shf{S}$.  Furthermore, if $\shf{S}$ is a presheaf of metric spaces then this correspondence is a homeomorphism.
\end{proposition}
\begin{proof}
To prove the first statement, it suffices to construct the other components in the product $\prod_{U\in\col{T}}\shf{S}(U)$ by applying each restriction $\shf{S}(U \subseteq X)$ to $s$ for $U \in \col{T}$.  Since $d_U$ is a pseudometric, the distance between each such component will vanish by construction.

Observe that this defines an injective function $i$ taking global sections to 0-approximate sections.  To prove the second statement, we need only construct an inverse to $i$ when $\shf{S}$ is a presheaf of metric spaces.  To that end, suppose that $s$ is instead a 0-approximate section of $\shf{S}$.  Since $s$ is an element of the product $\prod_{U\in\col{T}}\shf{S}(U)$, we can apply the projection
\begin{equation*}
j: \prod_{U\in\col{T}}\shf{S}(U) \to \shf{S}(X)
\end{equation*}
to obtain a global section $j(s)=s(X)$ since $X \in \col{T}$.  It is clear that $j \circ i$ is the identity function.  Since we have that $s$ is a 0-approximate section, then 
\begin{eqnarray*}
0 &=& d_V(s(V),\shf{S}(V\subseteq X)s(X))\\
&=&d_V(s(V),\shf{S}(V\subseteq X)j(s))\\
&=&d_V(s(V),i(j(s))_V)\\
\end{eqnarray*}
for all $V \in \col{T}$, which implies that $s(V) = i(j(s))_V$ because $d_V$ is a metric.  Therefore $i\circ j$ is the identity function, and the proof is complete.
\end{proof}

\begin{remark}
The technical distinction between the 0-approximate sections of presheaves of metric spaces and presheaves of pseudometric spaces runs quite deep.  Since the space of 0-approximate sections is closed, the subspace of global sections is also closed in a presheaf of metric spaces, but not necessarily so in a presheaf of pseudometric spaces.  Consider the topology $\col{T}=\{\emptyset,\{1\},\{1,2\}\}$ on the two-point set $X=\{1,2\}$ and the presheaf $\shf{P}$ defined by the diagram
\begin{equation*}
\xymatrix{
\shf{P}(\{1,2\}) = \{a,b\} \ar[r]^{\id} & \shf{P}(\{1\}) = \{a,b\}.
}
\end{equation*}
There exactly two global sections of $\shf{P}$, yet if the trivial pseudometric\footnote{All attributes are distance 0 apart in the \emph{trivial pseudometric}.} is applied to all sensor domains, then there are four 0-approximate sections.  Each of these approximate sections are limit points of the subspace of global sections, which shows that it is not closed.  (As the reader may verify, $\shf{P}$ also satisfies Axiom 6, it is therefore a sheaf.)
\end{remark}

The consistency radius of an assignment is an obstruction to it being a global section, as the following Proposition asserts.

\begin{proposition}
  If $a$ is an $\epsilon$-approximate section of a presheaf $\shf{P}$ of pseudometric spaces whose restrictions are Lipschitz continuous functions with Lipschitz constant $K$, then the distance (Definition \ref{df:assignment_pmetric}) between $a$ and the nearest global section to $a$ is at least $\frac{\epsilon}{1+K}$.
\end{proposition}

Assignments that are highly consistent (they have low consistency radius) are therefore easier to fuse via an optimization like Problem \ref{prob:fusion}.

\begin{proof}
  Suppose that $s$ is any global section of $\shf{P}$ and that $V \subseteq U \in \col{T}$ are open sets.  Then the consistency radius is the supremum of all distances of the form
  \begin{eqnarray*}
    d_V(a(V),\shf{P}(V \subseteq U)a(U)) &\le & d_V(a(V),\shf{P}(V \subseteq U)s(U)) + \\ && d_V(\shf{P}(V\subseteq U)s(U),\shf{P}(V\subseteq U)a(U))\\
    &\le&d_V(a(V),s(V)) + d_V(s(V),\shf{P}(V \subseteq U)s(U)) + \\ && d_V(\shf{P}(V\subseteq U)s(U),\shf{P}(V\subseteq U)a(U))\\
    &\le&d_V(a(V),s(V)) + d_V(\shf{P}(V\subseteq U)s(U),\shf{P}(V\subseteq U)a(U))\\
    &\le&(1+K)d_V(a(V),s(V)).
  \end{eqnarray*}
  Thus, taking supremums of both sides yields
  \begin{equation*}
    \epsilon \le (1+K) D(a,s),
  \end{equation*}
  which completes the argument.
\end{proof}

\begin{table}
  \begin{center}
        \caption{Table of observations for Example \ref{eg:sar_data}.}
    \label{tab:sar_data}

    \begin{tabular}{|c|c|c|c|c|c|}
      \hline
      Sensor&Entity&Case 1&Case 2& Case 3&Units\\
      \hline
      Flight plan&$x$&70.662&70.663&70.612&$^\circ$W\\
      &$y$&42.829&42.752&42.834&$^\circ$N\\
      &$z$&11178&11299&11237&m\\
      \hline
      ATC&$x$&70.587&70.657&70.617&$^\circ$W\\
      &$y$&42.741&42.773&42.834&$^\circ$N\\
      &$z$&11346&11346&11236&m\\
      &$v_x$&-495&-495&-419&km/h W\\
      &$v_y$&164&164&310&km/h N\\
      \hline
      RDF 1&$\theta_1$&77.1&77.2&77.2&$^\circ$ true N\\
      &$t$&0.943&0.930&0.985&h\\
      \hline
      RDF 2&$\theta_2$&61.3&63.2&63.3&$^\circ$ true N\\
      &$t$&0.890&0.974&1.05&h\\
      \hline
      Sat&$s$&(image)&&&\\
      &$s_x$&64.599&64.630&62.742&$^\circ$W\\
      &$s_y$&44.243&44.287&44.550&$^\circ$N\\
      \hline
      Field&$x$&70.649&70.668&70.626&$^\circ$W\\
      &$y$&42.753&42.809&42.814&$^\circ$N\\
      &$z$&11220&11431&11239&m\\
      &$v_x$&-495&-495&-419&km/h W\\
      &$v_y$&164&164&311&km/h N\\
      &$t$&0.928&1.05&1.02&h\\
      \hline
      Crash est.&$x$&65.0013&64.2396&65.3745&$^\circ$W\\
      &$y$&44.1277&44.3721&45.6703&$^\circ$N\\
      \hline
      Consist. rad.&&15.7&11.6&152&km\\
      Error&&16.1&17.3&193&km\\
      \hline
    \end{tabular}
  \end{center}
\end{table}

\begin{table}
  \begin{center}
        \caption{Fused results for Example \ref{eg:sar_data}}
    \label{tab:sar_data_fused}

    \begin{tabular}{|c|c|c|c|c|c|}
      \hline
      Sensor&Entity&Case 1&Case 2& Case 3&Units\\
      \hline
      Field&$x$&70.9391&71.8296&75.7569&$^\circ$W\\
      &$y$&42.7849&42.7806&42.5831&$^\circ$N\\
      &$z$&10963&11730&12452&m\\
      &$v_x$&-493.6&-497.2&-436.6&km/h W\\
      &$v_y$&168.6&162.2&280.9&km/h N\\
      &$t$&0.952&1.04&0.872&h\\
      \hline
      Crash est.&$x$&65.1704&65.4553&71.0939&$^\circ$W\\
      &$y$&44.2307&44.3069&44.7919&$^\circ$N\\
      \hline
      Error&&2.01&8.38&74.4&km\\
      \hline
    \end{tabular}
  \end{center}
\end{table}

\begin{figure}
  \begin{center}
    \includegraphics[width=5in]{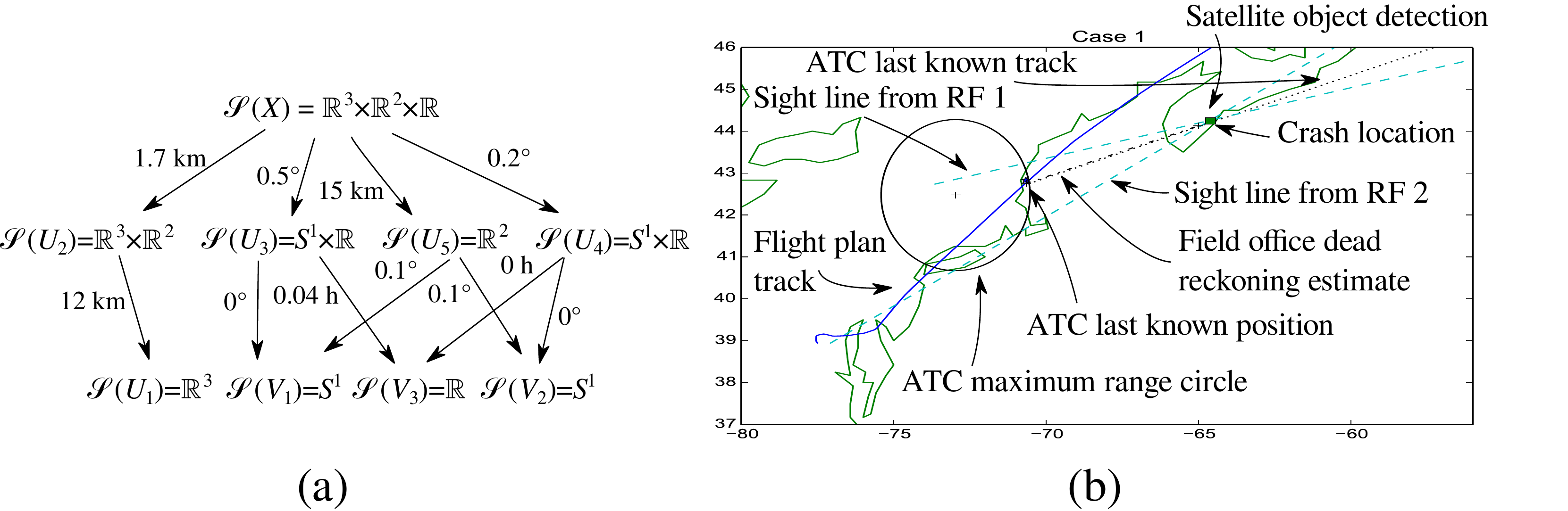}
    \caption{Case 1 in Example \ref{eg:sar_data}: (a) Errors in the presheaf model (b) Spatial layout.}
    \label{fig:sar_data1}
  \end{center}
\end{figure}

\begin{figure}
  \begin{center}
    \includegraphics[width=5in]{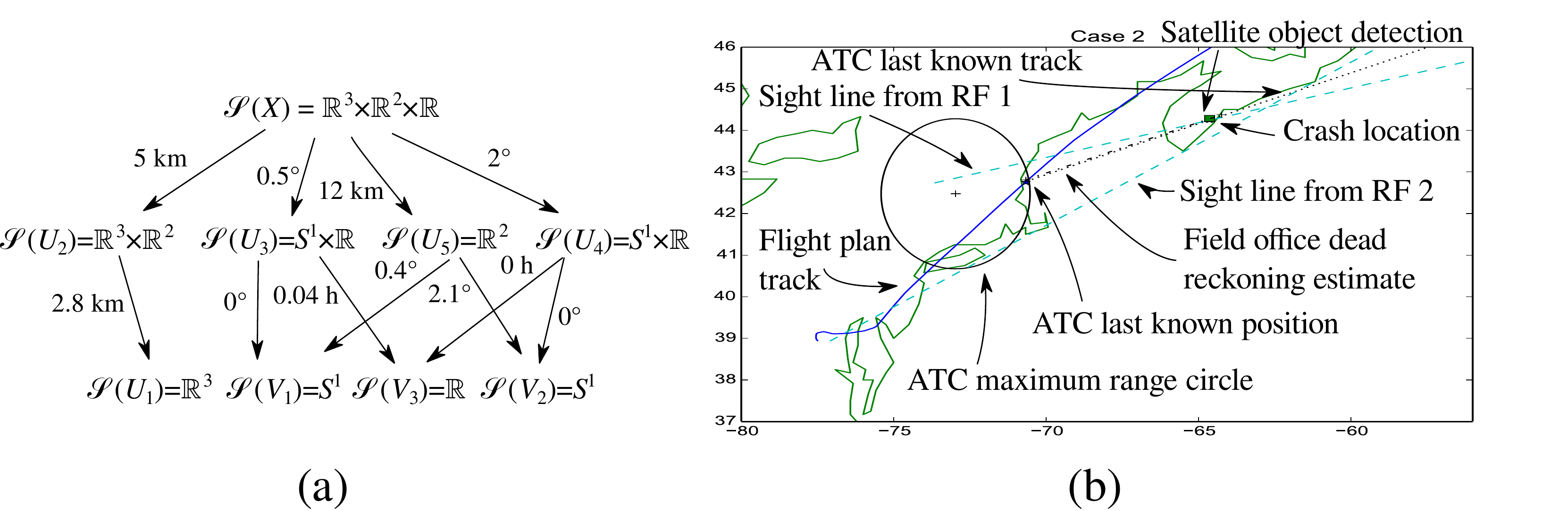}
    \caption{Case 2 in Example \ref{eg:sar_data}: (a) Errors in the presheaf model (b) Spatial layout.}
    \label{fig:sar_data2}
  \end{center}
\end{figure}

\begin{figure}
  \begin{center}
    \includegraphics[width=5in]{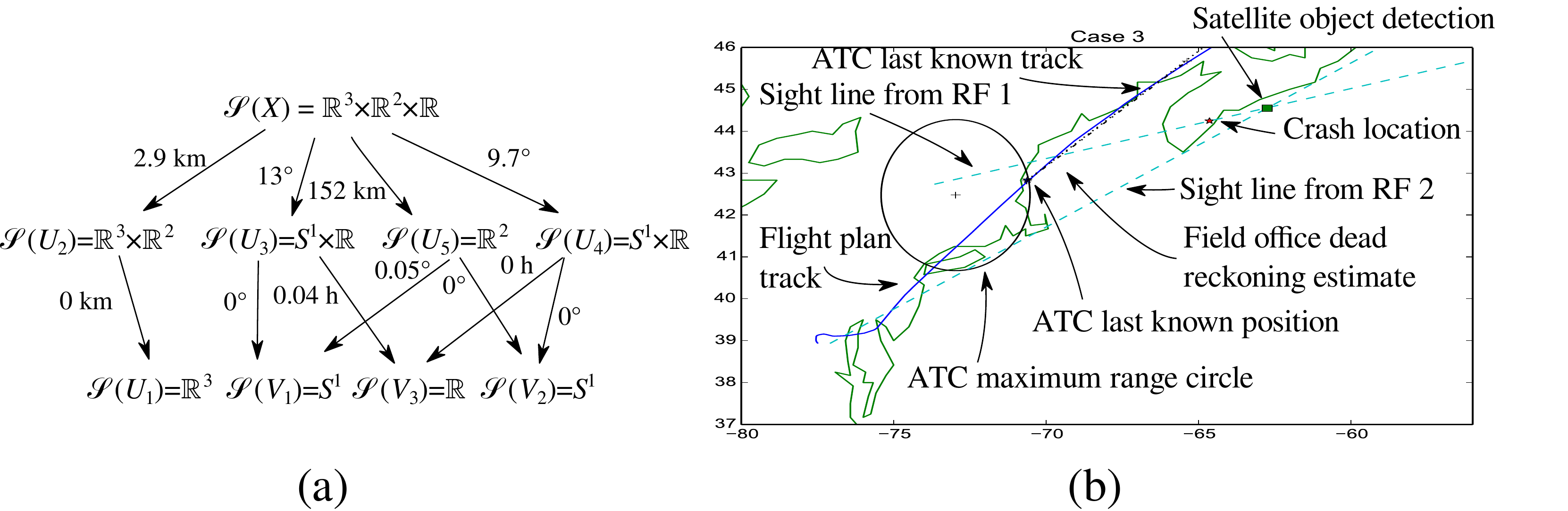}
    \caption{Case 3 in Example \ref{eg:sar_data}: (a) Errors in the presheaf model (b) Spatial layout.}
    \label{fig:sar_data3}
  \end{center}
\end{figure}

\begin{example}
  \label{eg:sar_data}
  Consider again the search and rescue scenario, recalling that we have constructed a sheaf model for the integrated sensor system in Examples \ref{eg:sar_rest} and \ref{eg:sar_full_def}.  The data shown in Table \ref{tab:sar_data} and Figures \ref{fig:sar_data1}--\ref{fig:sar_data3} were simulated to represent typical values for the sensors we deployed.  To constitute the flight plan, we used an actual flight path \cite{LH417} for Lufthansa flight 417.  Locations and operating ranges of ATC radars were taken from Air Force document AFD-090615-023.  The crash location was specified manually at $44.24545^\circ$ N, $64.63672^\circ$ W.  The crash location was withheld from the data presented to the sheaf model, so that the output of its data fusion could be compared to the true value.  The two RDF sensors were placed at
\begin{itemize}
\item The home location for the Rensselaer Polytechnic Amateur radio club station W2SZ (the author was formerly a member) $42.7338328^\circ$ N, $73.662574^\circ$ W, in upstate New York and
\item The author's office at the American University Mathematics and Statistics department $38.9352387^\circ$ N, $77.0897374^\circ$ W in downtown Washington, DC.
\end{itemize}
We simulated three distinct cases, in which different errors were applied to the data.  Measurements were perturbed by random errors that are representative of typical sensors.

We then solved Problem \ref{prob:fusion} by searching over the possible global sections to find the one closest (in the pseudometric from Definition \ref{df:assignment_pmetric}) to each case.  We used the Nelder-Mead simplex algorithm from the {\tt scipy.optimize} \cite{scipyopt} toolbox for this optimization.  The results are shown in Table \ref{tab:sar_data_fused}.  The error between the predicted crash location (using dead reckoning from the last known position and velocity) improved rather dramatically by fusing the data.

\begin{description}
\item[Case 1:] (See Figure \ref{fig:sar_data1}) This case was constructed so that there was only small error present in all measurements.  Applying dead reckoning to the Field office's data results in a crash location estimate that is $16.1$ km from the actual crash location.  Examining Figure \ref{fig:sar_data1}(b), it is clear that a crucial factor in obtaining this estimate was that the ATC measurement of the velocity vector and the RDF sensors' estimate of the crash time were both reliable.  Figure \ref{fig:sar_data1}(a) shows that consistency radius of the assignment corresponding to this Case is about $15.7$ km, which is fairly close to the actual error.  A closer examination shows that nearly all the error is concentrated in two places: between the ATC ($U_2$) and flight plan ($U_1$) and between the Field office ($X$) and the satellite detection ($U_5$).  The fused result exhibits markedly reduced error (a factor of $8\times$ reduced), largely since all of the errors were independent and unbiased. 

\item[Case 2:] (See Figure \ref{fig:sar_data2}) This case uses nearly the same setup as Case 1, but the RDF sensor 2 (from the author's office) was biased southward by $2^\circ$.  This error is immediately detectable in Figure \ref{fig:sar_data2}(a) as an anomalously high error between the Field office ($X$) and the RDF sensor 2 ($U_4$).  Since the only high errors occur between sensor domains involving RDF sensor 2 ($U_4$ and $V_2$), the sheaf model makes it clear that the errors are due exclusively to RDF sensor 2.  It should be noted that Figure \ref{fig:sar_data2}(a) does not include any \emph{a priori} knowledge of the sensor bias.  In this Case, the Field office did not use the RDF sensors for precise location -- only timing -- so its estimate by dead reckoning is only off by $17.3$ km from the actual crash location.  The consistency radius for this case is slightly smaller, at $11.6$ km.  This situation is a classic example of the power of data fusion, since the fused result shows that the error due to the biased RDF sensor is essentially eliminated by locating the nearest global section.  The fused result exhibits around a factor of $2\times$ reduction in error.

\item[Case 3:] (See Figure \ref{fig:sar_data3}) This case retains the RDF sensor 2 bias from Case 2, but also incorrectly records the last known velocity vector as being the one specified by the original flight plan.  This is fairly realistic as target heading estimates near the edge of the usable range of a radar are often quite noisy.  The consistency radius is easily read from Figure \ref{fig:sar_data3}(a) as $152$ km.  This is definitely an indication that something is amiss, given that the error incurred by dead reckoning is visible on Figure \ref{fig:sar_data3}(c) and is $193$ km.  Even so, the fused result is considerably closer to the actual crash location, exhibiting somewhat over a factor of $2\times$ lower error.
\end{description}

\end{example}

\section{Analysis of integrated sensor systems}
\label{sec:cech}

In his Appendix on sheaf theory, Hubbard states \cite{Hubbard} ``It is fairly easy to understand what a sheaf is, especially after looking at a few examples.  Understanding what they are good for is rather harder; indeed, without cohomology theory, they aren't good for much.''  Although the previous sections of this article advocate otherwise, Hubbard makes an important point.  Although there are many co-universal information representations (of which sheaves are one), the stumbling block for integration approaches is the lack of theoretically-motivated computational approaches.  What sheaf theory provides is a way to perform \emph{controlled algebraic summarization} -- a weaker set of invariants that still permit certain kinds of inference.  In order to perform these summarizations, the attributes must be specially encoded, which leads to our final Axiom.

\begin{description}
\item[Axiom 7:] Each space of observations $\shf{S}(U)$ has the structure of a Banach space (complete normed vector space) and each restriction $\shf{S}(U \subseteq V)$ is a continuous linear transformation of this Banach space.  We call $\shf{S}$ a \emph{sheaf of Banach spaces}.
\end{description}

\begin{proposition}
  \label{prop:banach_fusion}
  If $\shf{S}$ is a sheaf of Banach spaces on a finite topology $\col{T}$, then Problem \ref{prob:fusion} always has a unique solution, embodied as a projection from the space of assignments to the space of global sections.
\end{proposition}
\begin{proof}
  One need only observe that the metric $D$ on assignments (Definition \ref{df:assignment_pmetric}) makes the space of assignments into a Banach space.  Specifically, if $a \in \prod_{U \in \col{T}}\shf{S}(U)$ is an assignment of $\shf{S}$, let
\begin{equation*}
\|a\| = \max_{U \in \col{T}} \|a(U)\|_{\shf{S}(U)},
\end{equation*}
where $\|\cdot\|_{\shf{S}(U)}$ is the norm on $\shf{S}(U)$.  The maximum exists since $\col{T}$ is finite.  Then the metric $D$ is given by
\begin{eqnarray*}
\|a-b\| &=& \max_{U \in \col{T}} \|a(U)-b(U)\|_{\shf{S}(U)}\\
 &=& \max_{U \in \col{T}} d_U(a(U),b(U))\\
 &=& D(a,b).
\end{eqnarray*}
\end{proof}

There are a number of additional benefits once Axiom 7 is satisfied, for instance,
\begin{enumerate}
\item The space of global sections can be computed directly using linear algebra (Proposition \ref{prop:cohomology_global}), which greatly simplifies solving the Problem \ref{prob:fusion}, the data fusion problem,
\item There is a canonical collection of invariants, against which all others can be compared (Definition \ref{def:cech_cohomology}), and
\item We need not construct sets of observations over the unions of sensor domains in order to access these invariants (Theorem \ref{thm:leray}).
\end{enumerate}

\subsection{Preparation: Stochastic observations}
Although Axiom 7 usually does permit reasonable options for restrictions, it may also be highly constraining.  Many effective algorithms that could be used as restrictions are not linear.  Although this seems to dramatically limit the effectiveness of our subsequent analysis, there is a way out that actually increases the expressiveness of the model.  Specifically, a non-linear restriction map can be ``enriched'' into one that is linear \cite{RobinsonLogic, Purvine_JMM2016}.  The most concrete of these enrichments lifts the entities and attributes into a stochastic modeling framework, as the next Example shows.

\begin{figure}
  \begin{center}
    \includegraphics[width=4in]{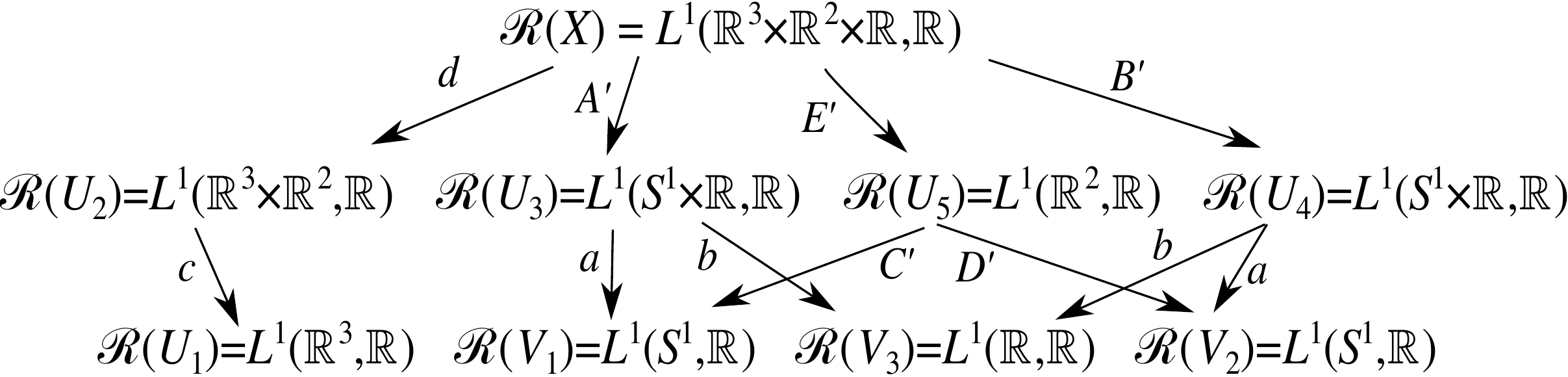}
    \caption{The sheaf in the search and rescue example written using linear stochastic maps.}
    \label{fig:sar_categorified}
  \end{center}
\end{figure}

\begin{figure}
  \begin{center}
    \includegraphics[width=4in]{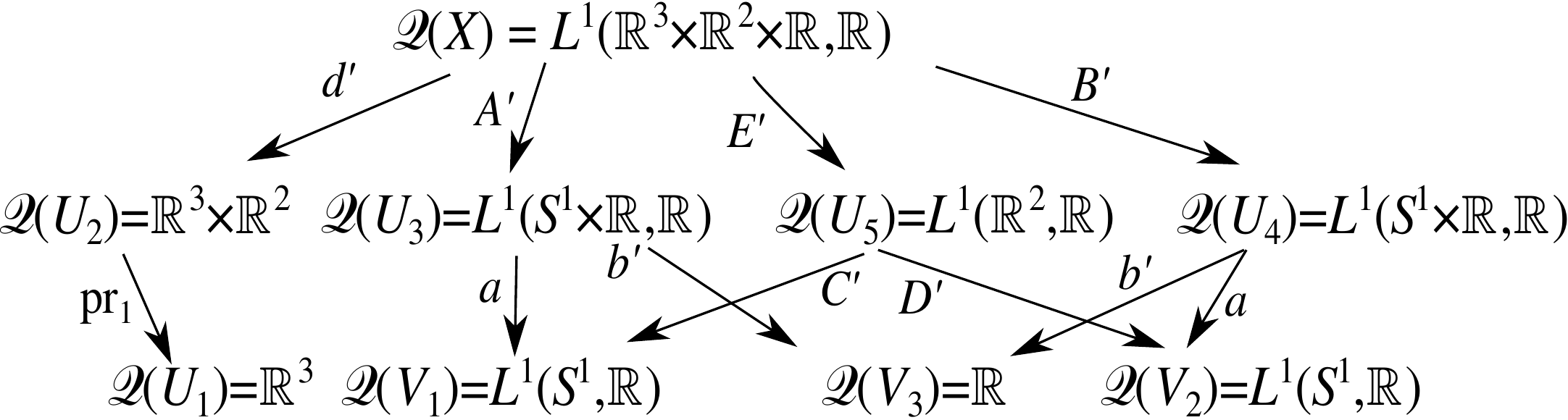}
    \caption{The sheaf in the search and rescue example written using linear stochastic maps and means.}
    \label{fig:sar_categorified2}
  \end{center}
\end{figure}

\begin{example}
  \label{eg:sar_categorified}
The sheaf model $\shf{S}$ defined for the search and rescue scenario in Example \ref{eg:sar_full_def} does not satisfy Axiom 7, since the restrictions $A$, $B$, $C$, $D$, and $E$ are not linear transformations.  A different sheaf model $\shf{R}$ can instead be defined that does satisfy Axiom 7, by reinterpreting the attributes as probability distributions.  The restrictions then all become stochastic maps, all of which are linear.  

The sheaf $\shf{R}$ is built upon the same topological space $(X,\col{T})$ as $\shf{S}$, and each space of observations over an open set open set $U\in\col{T}$ is replaced by the recipe
\begin{equation*}
\shf{R}(U) = L^1(\shf{S}(U),\mathbb{R}),
\end{equation*}
where $L^1(M,\mathbb{R})$ is the space of real-valued integrable functions on a measure space $M$.  We interpret $L^1(M,\mathbb{R})$ as the vector space generated by the probability distributions on $M$.  

The appropriate sheaf diagram of $\shf{R}$ is shown in Figure \ref{fig:sar_categorified}, for which we now describe the restrictions.  Every restriction of $\shf{S}$ translates into a linear, stochastic map in $\shf{R}$ according to the following procedure, which we state generally first and then specialize.  Using an arbitrary measurable function $f:M \to N$, we can define a function $F:L^1(M,\mathbb{R}) \to L^1(N,\mathbb{R})$ by the following formula for each $u\in L^1(M,\mathbb{R})$ and $y\in N$
\begin{equation}
\label{eq:lift}
\left(F(u)\right)(y) = \int_{f^{-1}(\{y\})} u(x) \; dx = \int_M u(x) \delta(y-f(x))\; dx,
\end{equation}
where $\delta$ is the Dirac distribution.  If the volume of $M$ is 1, so that
\begin{equation*}
\int_M dx = 1,
\end{equation*}
then \eqref{eq:lift} is precisely Bayes' rule, and $F$ can be shown to be a stochastic linear map acting on $u$, by an appeal to Fubini's theorem.  Specifically, if $u$ is nonnegative and has integral 1, then $F(u)$ will also have these properties.

In the case of the restrictions $a$--$d$ in Figure \ref{fig:sar_categorified} for $\shf{R}$, all of which are projectors of various sorts, Equation \ref{eq:lift} specializes to
\begin{eqnarray*}
\left(a (u)\right) (\theta) &=& \int_{-\infty}^\infty u(\theta,t) \; dt,\\ 
\left(b (u)\right) (t) &=& \int_{0}^{2\pi} u(\theta,t) \; d\theta,\\ 
\left(c (u)\right) (x,y,z)&=& \int_{-\infty}^\infty \int_{-\infty}^\infty u(x,y,z,v_x,v_y) \; dv_x\, dv_y, \text{ and}\\
\left(d (u)\right)(x,y,z,v_x,v_y) &=& \int_{-\infty}^\infty u(x,y,z,v_x,v_y,t)\; dt,\\
\end{eqnarray*}
where we have used $u$ to represent a typical function in the domain of each restriction.

Of course, we must also define $A'$, $B'$, $C'$, $D'$, and $E'$ using Equation \ref{eq:lift}, replacing $f$ with each restriction from $\shf{S}$, so we obtain
\begin{eqnarray*}
  \left(A'(u)\right)(\theta_1,t) &=& \int\int\int\int_0^\infty u(r\sin\theta_1 + r_{1x}-v_xt,r\cos\theta_1+r_{1y}-v_yt,z,v_x,v_y,t)\; r\, dr\, dz\,dv_x\,dv_y,\\
\left(B'(u)\right)(\theta_2,t) &=& \int\int\int\int_0^\infty u(r\sin\theta_2 + r_{2x}-v_xt,r\cos\theta_2+r_{2y}-v_yt,z,v_x,v_y,t)\; r\, dr\, dz\,dv_x\,dv_y,\\
  \left(C'(u)\right)(\theta_1) &=& \int_0^\infty u(r\sin\theta_1 + r_{1x},r\cos\theta_1+r_{1y})\; r\, dr,\\
  \left(D'(u)\right)(\theta_2) &=& \int_0^\infty u(r\sin\theta_2 + r_{2x},r\cos\theta_2+r_{2y})\; r\, dr,\\
  \left(E'(u)\right)(s_x,s_y) &=& \int\int\int\int u(s_x-v_xt,s_y-v_yt,z,v_x,v_y,t)\;dz\,dv_x\,dv_y\,dt,
\end{eqnarray*}
where $(s_x,s_y)$ are coordinates of an object detected in the satellite image, $(r_{1x},r_{1y})$ are the coordinates of the first RDF sensor and $(r_{2x},r_{2y})$ are the coordinates of the second RDF sensor.

It is not strictly necessary to retain probability distributions everywhere in the sheaf model.  Portions of the sheaf in which there were already linear restrictions can be retained undisturbed by extracting the mean values of the distributions.  This process is not unique, but one possibility is shown in Figure \ref{fig:sar_categorified2}, where 
\begin{eqnarray*}
b'(u) &=& \int_{-\infty}^\infty t \left(b(u)\right)(t)\; dt \\
&=& \int_{-\infty}^\infty \int_{0}^{2\pi} t u(\theta,t) \; d\theta\, dt,\\ 
\left(d' (u)\right)(x,y,z,v_x,v_y) &=& \left( \int\int\int\int\int x \left(d (u)\right)(x,y,z,v_x,v_y)\; dx\,dy\,dz\,dv_x\,dv_y, \dotsc\right).
\end{eqnarray*}
\end{example}

\subsection{Computing cohomology}

The sheaf $\shf{S}$ of vector spaces that results from Axiom 7 has a canonical family of invariants called \emph{cohomology groups} $H^k(\shf{S})$ for $k=0,1,\dotsc$.  In the most general setting, the cohomology groups can be constructed in various ways using linear algebraic manipulations, most notably the \emph{canonical} or \emph{Godement resolution} \cite{Godement_1958}\cite[II.2]{Bredon}.  Instead, we will use a more practical construction, called \emph{\v{C}ech cohomology} \cite{Hubbard,Bredon}.

\begin{definition}
  \label{def:cech_cohomology}
  Suppose $\col{U}=\{U_1, \dotsc, U_n\}$ is a finite open cover for a topological space $(X,\col{T})$, which means that $\col{U} \subseteq \col{T}$ and $\cup \col{U} = X$.  Given a sheaf $\shf{S}$ of Banach spaces on $\col{T}$, the \emph{$k$-cochain spaces of $\shf{S}$ subject to $\col{U}$} are given by the Cartesian products of the stalks over $(k+1)$-way intersections of sets in $\col{U}$
  \begin{equation*}
    C^k(\col{U};\shf{S}) = \prod_{i_0 < i_1 \dotsb i_k} \shf{S}(U_{i_0} \cap \dotsc U_{i_k}).
  \end{equation*}
  Notice that the coordinates of elements of $C^k(\col{U};\shf{S})$ are indexed by lists of $(k+1)$ increasing integers to eliminate duplications.  

Beware that the $k$ on $C^k$ is merely an index, and that $C^k$ does \emph{not} mean the space of $k$-times continuously differentiable functions!
  
  We also define the \emph{$k$-coboundary map}, which is a function $d^k : C^k(\col{U};\shf{S}) \to C^{k+1}(\col{U};\shf{S})$.  This function is defined by specifying its coordinates on an element $a \in C^k(\col{U};\shf{S})$.
  \begin{equation*}
    d^k(a)_{i_0 < \dotsb i_{k+1}} = \sum_{j=0}^{k+1} (-1)^j \shf{S}(U_{i_0} \cap \dotsb U_{i_{k+1}} \subseteq U_{i_0} \cap \dotsb \widehat{U_{i_j}} \dotsb U_{i_{k+1}}) a_{i_0 < \dotsb \widehat{i_j} \dotsb i_{k+1}},
  \end{equation*}
  where the hat indicates omission from a list.

  It is a standard fact that $d^{k+1} \circ d^k = 0$, from which we can make the definition of \emph{the $k$-cohomology spaces}
  \begin{equation*}
    H^k(\col{U};\shf{S}) = \text{ker } d^k / \text{image } d^{k+1}.
  \end{equation*}
  We sometimes write $H^k(\shf{S})$ for $H^k(\col{T};\shf{S})$ to reduce notational clutter.
\end{definition}

Although computing $H^k(\col{T};\shf{S})$ is a bit difficult, simply due to the number of dimensions of $C^k(\col{T};\shf{S})$, the first consequence of this definition is that the zeroth cohomology classes are the global sections.  

\begin{proposition} \cite[Thm 4.3]{TSPbook}
  \label{prop:cohomology_global}
  If $\shf{S}$ is a sheaf of vector spaces on a finite topology $\col{T}$, then  $H^0(\col{T};\shf{S})=\ker d^0$ consists precisely of those assignments $s$ which are global sections of $\shf{S}$.
\end{proposition}

Since the $C^0(\col{T};\shf{S})$ is also the space of assignments, Proposition \ref{prop:banach_fusion} implies that the solution to the Data Fusion Problem (Problem \ref{prob:fusion}) is merely a projection $C^0(\col{T};\shf{S}) \to H^0(\col{T};\shf{S})$.

Unfortunately, due to the way that $H^0(\col{T};\shf{S})$ is constructed, $\shf{S}(X)$ is already a factor in $C^0(\col{T};\shf{S})$.  This is remedied by the Leray theorem, which allows for a considerable reduction in complexity by verifying whether a smaller cover $\col{U}$ than the whole topology $\col{T}$ still recovers all of the cohomology spaces.

\begin{theorem} (Leray theorem, \cite{Leray_1950} \cite[Thm A7.2.6]{Hubbard}\cite[Thm 4.13 in III.4]{Bredon})
  \label{thm:leray}
  Suppose $\shf{S}$ is a sheaf on $\col{T}$, and that $\col{U}\subseteq \col{T}$ is a collection of open sets.   If for every intersection $U=U_0 \cap \dotsb U_k$ of a collection of elements in $\col{U}$ it follows that
  \begin{equation*}
    H^k(U \cap \col{T};\shf{S})=0 \text{ for all }k>0,
  \end{equation*}
  then we can conclude that $H^p(\col{U};\shf{S}) \cong H^p(\col{T};\shf{S})$ for all $p$.
\end{theorem}

Practically speaking, the nontrivial elements of $H^k$ for $k>0$ represent obstructions to consistent sensor data.  They indicate situations in which it is impossible for an assignment to some open sets to be extended to the rest of the topology.  Reading from Definition \ref{def:cech_cohomology}, a nontrivial element of $H^k(\col{U};\shf{S})$ consists of a joint collection of observations on the $k$-way intersections of sensor domains in $\col{U}$ that are consistent upon further restriction to $(k+1)$-way intersections, but do not arise from any joint observations from $(k-1)$-way intersections.  They are therefore certain classes of partially self-consistent data that cannot be extended to global sections.  From a data fusion standpoint, they may be considered solutions to Problem \ref{prob:fusion} on only part of the space, but these solutions cannot be extended to all of the space without incurring distortion.

When situations with nontrivial $H^k$ for $k>0$ arise in practice they usually indicate a failure of the model to correspond to reality.  For instance, in all of the mosaic-forming examples discussed in this article, there has been an implicit assumption that all images were acquired at the same time.  When time evolves between imaging events, nontrivial $H^k$ elements indicate that inconsistency can easily arise, as Example \ref{eg:h1_generator} shows.

\begin{figure}
\begin{center}
\includegraphics[width=4in]{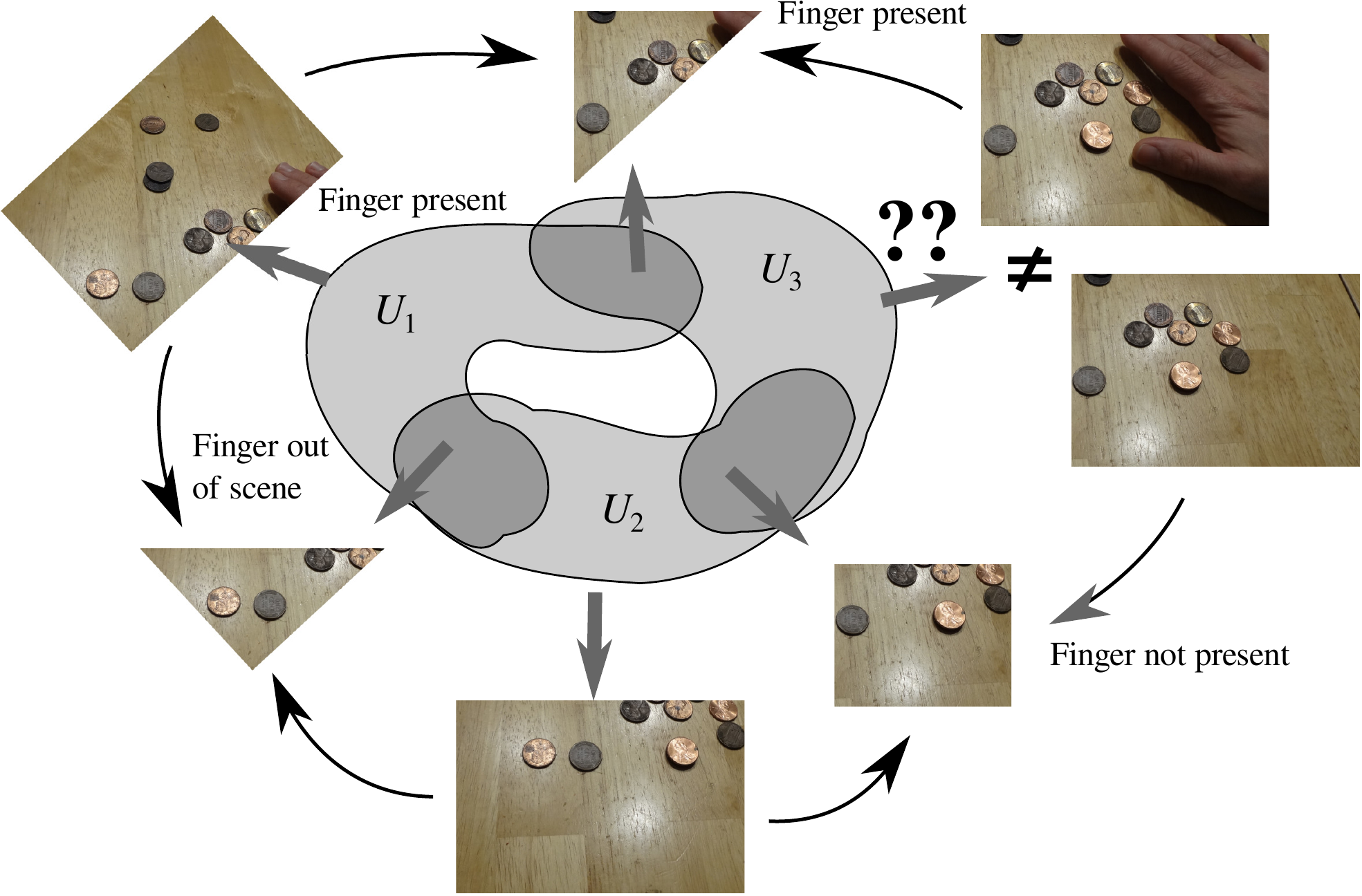}
\caption{A collection of images in a nontrivial $H^1$ cohomology class.}
\label{fig:h1_generator}
\end{center}
\end{figure}

\begin{example}
  \label{eg:h1_generator}
Consider the collection of photographs shown in Figure \ref{fig:h1_generator}.  Consider the sheaf $\shf{S}$ for the three sensor domains $\col{U}=\{U_i\}$ as shown, where the observations consist of various spaces of images and the restrictions consist of image cropping operations, as in Examples \ref{eg:cam_top} and \ref{eg:coins}.

Consider the two assignments consisting of the photograph on the left (on $U_1$) and the bottom ($U_2$) and one of the photographs on the right ($U_3$).  Neither of these assignments correspond to a global section, because one of the restriction maps out of $U_3$ will result in a disagreement on one of its intersections with the other two sets.  In particular, the set of photographs shown on the intersections could not possibly have arisen by restricting \emph{any} assignment of photographs to the elements of the cover, that is, the image through $d^1$ of an element of $C^0(\col{U}; \shf{S})$.  Therefore, the assignment of photographs to the intersections shown in Figure \ref{fig:h1_generator} is an example of a nontrivial class in $H^1(\col{T};\shf{S})$ since there are no nonempty intersections between all three sensor domains.
\end{example}

In light of Definition \ref{def:cech_cohomology} and the example, we may be tempted to add more sensors to our collection to try to resolve the ambiguity.  The essential fact is that without bringing time into the model of mosaic-forming, problems still can arise.  Under fairly general conditions, such as those in the above example, the collection of sensors contains all of the possible information as is truly encoded by the model.

\begin{figure}
\begin{center}
\includegraphics[width=3in]{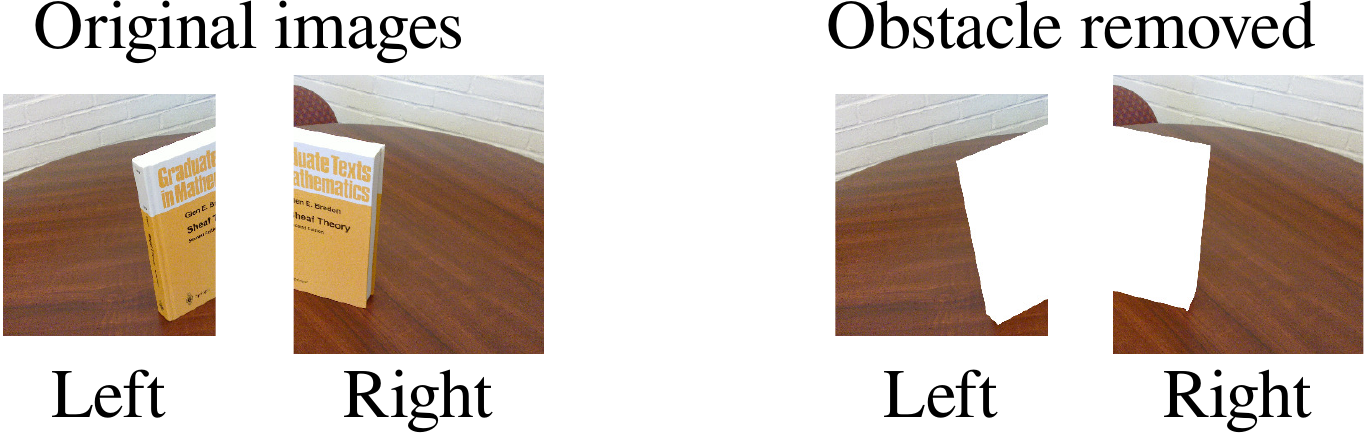}
\caption{Trying to see behind a book with two different views.}
\label{fig:obstacle}
\end{center}
\end{figure}

\begin{example}
  \label{eg:obstacle}
Consider the case of trying to use a collection of sensors to piece together what is behind an obstacle.  This occurs in optical systems as well as those based on radio or sound propagation.  Usually, the process involves taking collections of images from several angles followed by an obstacle-removal step.  If not enough looks are obtained from enough angles, then it will not be possible to infer anything from behind the obstacle.  However, if the sensors' locations are not well known, it may be impossible to detect that this has occurred.  This is an effect of a difference between the sensors' view of $H^1(\col{U};\shf{S})$ and the actual $H^1(\col{T};\shf{S})$.  

Figure \ref{fig:obstacle} shows the process for trying to see behind a book using two views.  The original photographs on the left show the scene, while the photos on the right show the obstacle removed.  Suppose that $X$ consists of the following entities: $L$ for the left camera, $R$ for the right camera, $V_1$ for the upper overlap, and $V_2$ for the lower overlap.  That is, 
\begin{equation*}
X=\{R,L,V_1,V_2\}.
\end{equation*}
There is a natural topology for the sensors, namely
\begin{equation*}
  \col{T}=\{\{R,V_1,V_2\},\{L,V_1,V_2\},\{V_1,V_2\},\{V_1\},\{V_2\},X,\emptyset\}
\end{equation*}
in which the overlaps are separated.  Modeling the cameras as $\col{U} = \{U_L,U_R\}$ where $U_L = \{L,V_1,V_2\}$, $U_R = \{R,V_1,V_2\}$, we obtain the following sheaf $\shf{M}$ associated to the topology generated by $\col{U}$:
\begin{equation*}
  \xymatrix{
&X&    &&\mathbb{R}^{3m}\oplus\mathbb{R}^{3n}\ar[dr]\ar[dl]&\\
U_L\ar[ur]&\{V_1,V_2\}\ar[r]\ar[l]&U_R\ar[ul]    &\mathbb{R}^{3m} \ar[r] & \mathbb{R}^{3p}\oplus\mathbb{R}^{3q} & \mathbb{R}^{3n} \ar[l]
    }
\end{equation*}
where $m$ is the number of pixels in the left image, $n$ is the number of pixels in the right image, $p$ is the number of pixels in the upper overlap region, and $q$ is the number of pixels in the lower overlap region.  Since the overlap region consists of two connected components, does this impact the mosaic process?  It does not, according to Theorem \ref{thm:leray}, though this requires a little calculation.  It is clear the only place we need to be concerned is the overlap.  So we need to examine the intersection $U_L \cap U_R = \{V_1,V_2\}$ and compute
\begin{equation*}
  H^k(\col{V};\shf{M}),
\end{equation*}
where $\col{V} = \{\{V_1,V_2\},\{V_1\},\{V_2\}\}$.  The sheaf diagram for this (with the sets of $\col{V}$ below for clarity) is
\begin{equation*}
  \xymatrix{
    \mathbb{R}^{3p} & \mathbb{R}^{3p}\oplus\mathbb{R}^{3q} \ar[l]^{\pr_1}\ar[r]^{\pr_2} & \mathbb{R}^{3q}\\
    \{V_1\} \ar[r]& \{V_1,V_2\} & \{V_2\} \ar[l]\\
    }
\end{equation*}
which is easily computed to have $\dim H^k(\col{V};\shf{M}) = 0$ for $k>0$.  (In the diagram, $\pr_k$ is the projection onto the $k$-th factor of a product and $\id$ is the identity map.)
\end{example}

The above example shows that solving a data fusion problem applied to relatively unprocessed data works well.\footnote{It works well up to the limit of having enough agreement between the sensors to deduce the overlaps, of course.  This is not necessarily easy!}  It also works if data are processed a bit more, as the next (and final) example shows.

\begin{figure}
  \begin{center}
    \includegraphics[width=3in]{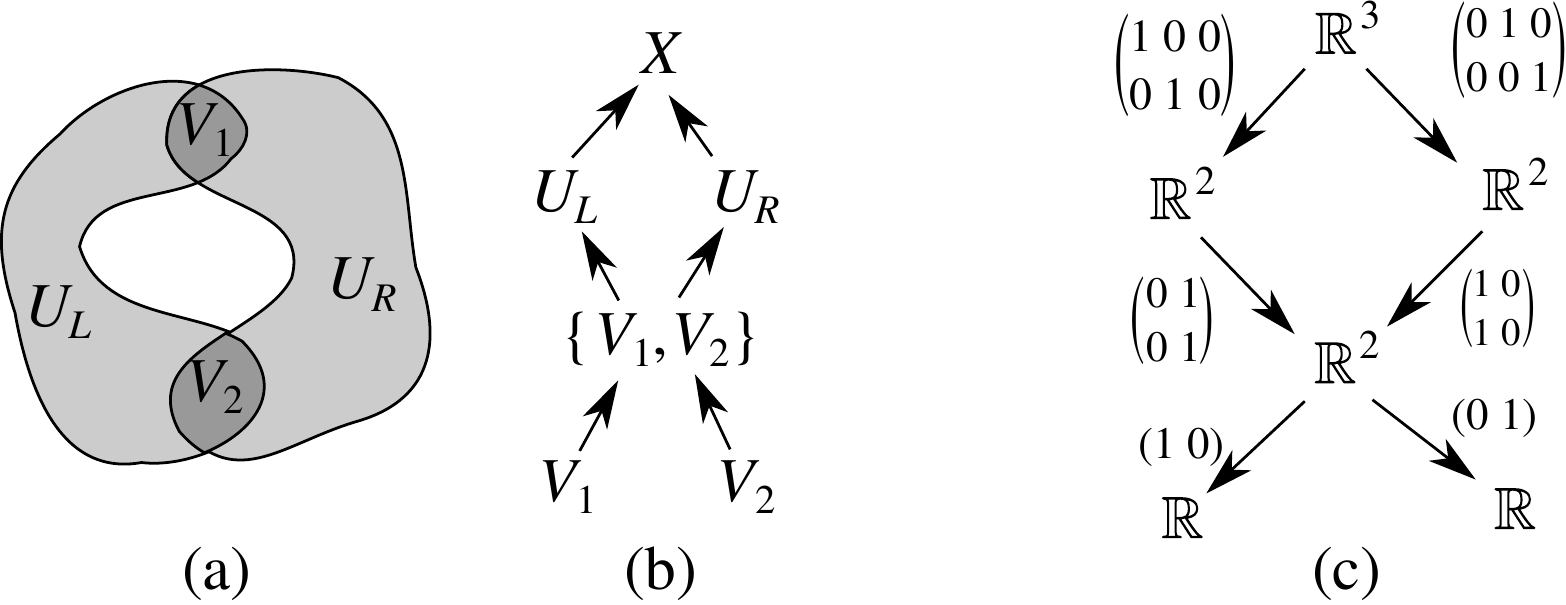}
    \caption{The sheaf for Example \ref{eg:obstacle_sheaf}.  (a) A representation of the topological space $\col{T}$ in Examples \ref{eg:obstacle} and \ref{eg:obstacle_sheaf}.  (b) The open sets in the topology. (c) The sheaf $\shf{P}$ and its restrictions.}
    \label{fig:obstacle_sheaf}
  \end{center}
\end{figure}

\begin{example}
  \label{eg:obstacle_sheaf}
  Suppose instead that we are interested not in the images themselves but rather whether a single object is in the scene.  Like the previous example the scene can be divided into several regions, each represented by an entity with the topology $\col{T}$.  If we attempt to quantify the probability that the object is in the left, right, or middle of the scene, the space of observations over $U_L$ (or $U_R$) is $\mathbb{R}^2$, whose two components should be interpreted as the probability that an object is present on the left (or right) and in the overlap.  We are led to consider a sheaf $\shf{P}$ on $\col{T}$ that keeps track of whether there is an object on the left, right, or middle of the scene as shown in Figure \ref{fig:obstacle_sheaf}(c).  As before, if $\col{U} = \{U_L,U_R\}$ where $U_L = \{L,V_1,V_2\}$, $U_R = \{R,V_1,V_2\}$, then $\dim H^0(\col{U};\shf{P}) = 3$, indicating the three probability values for the left, right, and middle of the scene.  However, unlike the Example \ref{eg:obstacle}, $\dim H^1(\col{U};\shf{P})=1$, which indicates that there can be data that are locally self-consistent, but not globally self-consistent.  This corresponds to the situation where there are different probabilities for the object's presence in $V_1$ and $V_2$ -- both variants of the middle region -- yet the global model only discerns a single probability for the middle region.

  We can confirm that this is not an artifact by examining the refined cover  $\col{V} = \{\{V_1,V_2\},\{V_1\},\{V_2\}\}$ as in the previous example, and are once again led to the fact that $\dim H^k(\col{V};\shf{P})=0$ for $k>0$.
\end{example}

The presence of nontrivial cohomology in $H^1$ arising after performing some data processing indicates that the processing has lost information at a global level, but that this information is still possible to detect with individual sensors.  This strongly advocates for performing fusion as close to the sensors as possible, to avoid unexpected correlations across sensor domains due to inconsistent processing.  This view is shared by others, for instance \cite{Newman_2013}, though in some cases the artifacts can be anticipated.  For instance, \cite{basu2003computing} shows how to use spectral sequences to compensate for when the hypotheses of Theorem \ref{thm:leray} are known to fail.  As a special case for graphs, \cite{RobinsonQGTopo} and \cite[Sec 4.7]{TSPbook} contain an algorithmic approach that is somewhat simpler.

\section{Conclusions and new frontiers}

This article has demonstrated that sheaves are an effective theoretical tool for constructing and exploiting integrated systems of sensors in a bottom-up, data-driven way.  By constructing sheaf models of integrated sensor systems, one can use them as a basis for solving data fusion problems driven by sensor data.  One can also apply cohomology as a tool to identify features of the integrated system as a whole, which could be useful for sensor deployment planning.  

There are several mathematical frontiers that are opened by repurposing sheaf theory for information integration.  For instance, although we have focused on global sections and obstructions to them, the combinatorics of \emph{local} sections is essentially not studied.  The extension of observations to global sections arises in many important problems.  In homogeneous data situations, such as the reconstruction of signals from samples \cite{RobinsonSampTA2013} and the solution of differential equations \cite{Ehrenpreis_1956}, sheaves are already useful.  The generality afforded by sheaves indicates that problems of this sort can also be addressed when the data types are not homogeneous.  Although persistent homology has received considerable recent attention in exploratory data analysis, persistent cohomology \cite{DeSilva_2011} has shown promise in signal processing and is a sheaf \cite{TSPbook}.

Although the examples in this article were large enough to be useful on their own, there is also considerable work to be done to address larger collections of sensors.  Although solving the data fusion problem is straightforward computationally, computing cohomology can require significant computing resources.  In homology calculations, the notion of collapses and reductions are essential for making computation practical.  A similar methodology -- based on discrete Morse theory \cite{Nanda_2015} -- exists for computing sheaf cohomology.  In either case (reduced or not), efficient finite field linear algebra could be an enabler. Finally, practical categorification \cite{Purvine_JMM2016} is still very much an art, and has a large impact on practical performance of sheaf-based information integration systems that rely on cohomology.

\section*{Acknowledgment}
  The author would like to thank the anonymous referees for their insightful comments and questions.  This work was partially supported under DARPA SIMPLEX N66001-15-C-4040.   The views, opinions, and/or findings expressed are those of the authors and should not be interpreted as representing the official views or policies of the Department of Defense or the U.S. Government.

\bibliographystyle{ieeetr}
\bibliography{sheafcanon_bib}

\begin{thebibliography}{10}

\bibitem{Magnuson_2010}
S.~Magnuson, ``Military 'swimming in sensors and drowning in data','' {\em
  National Defense}, vol.~94, January 2010.

\bibitem{Hall_2008}
D.~L. Hall, M.~McNeese, J.~Llinas, and T.~Mullen, ``A framework for dynamic
  hard/soft fusion,'' in {\em Information Fusion, 2008 11th International
  Conference on}, pp.~1--8, IEEE, 2008.

\bibitem{White_1991}
F.~White, ``Data fusion lexicon,'' Tech. Rep. ADA529661, Joint Directors of
  Laboratories, Washington, DC, October 1991.

\bibitem{Wald_1999}
L.~Wald, ``Some terms of reference in data fusion,'' {\em IEEE Trans.
  Geosciences and Remote Sensing}, vol.~37, no.~3, pp.~1190--1193, 1999.

\bibitem{wald2002data}
L.~Wald, {\em Data fusion: definitions and architectures: fusion of images of
  different spatial resolutions}.
\newblock Presses des MINES, 2002.

\bibitem{khaleghi2013multisensor}
B.~Khaleghi, A.~Khamis, F.~O. Karray, and S.~N. Razavi, ``Multisensor data
  fusion: A review of the state-of-the-art,'' {\em Information Fusion},
  vol.~14, no.~1, pp.~28--44, 2013.

\bibitem{varshney1997multisensor}
P.~K. Varshney, ``Multisensor data fusion,'' {\em Electronics \& Communication
  Engineering Journal}, vol.~9, no.~6, pp.~245--253, 1997.

\bibitem{alparone2008multispectral}
L.~Alparone, B.~Aiazzi, S.~Baronti, A.~Garzelli, F.~Nencini, and M.~Selva,
  ``Multispectral and panchromatic data fusion assessment without reference,''
  {\em Photogrammetric Engineering \& Remote Sensing}, vol.~74, no.~2,
  pp.~193--200, 2008.

\bibitem{Zhang_2010}
J.~Zhang, ``Multi-source remote sensing data fusion: status and trends,'' {\em
  International Journal of Image and Data Fusion}, vol.~1, no.~1, pp.~5--24,
  2010.

\bibitem{Dawn_2010}
S.~Dawn, V.~Saxena, and B.~Sharma, ``Remote sensing image registration
  techniques: A survey,'' in {\em Image and Signal Processing: 4th
  International Conference, ICISP 2010, Trois-Rivi{\`e}res, QC, Canada, June
  30-July 2, 2010. Proceedings} (A.~Elmoataz, O.~Lezoray, F.~Nouboud,
  D.~Mammass, and J.~Meunier, eds.), pp.~103--112, Berlin, Heidelberg: Springer
  Berlin Heidelberg, 2010.

\bibitem{Koetz_2007}
B.~Koetz, G.~Sun, F.~Morsdorf, K.~Ranson, M.~Kneubühler, K.~Itten, and
  B.~Allgöwer, ``Fusion of imaging spectrometer and {LIDAR} data over combined
  radiative transfer models for forest canopy characterization,'' {\em Remote
  Sensing of Environment}, vol.~106, no.~4, pp.~449 -- 459, 2007.

\bibitem{Guo_2008}
Z.~Guo, G.~Sun, K.~Ranson, W.~Ni, and W.~Qin, ``The potential of combined
  {LIDAR} and {SAR} data in retrieving forest parameters using model
  analysis,'' in {\em Geoscience and Remote Sensing Symposium, 2008. IGARSS
  2008. IEEE International}, vol.~5, pp.~V -- 542--V -- 545, July 2008.

\bibitem{Chen_2014}
J.~Chen, H.~Feng, K.~Pan, Z.~Xu, and Q.~Li, ``An optimization method for
  registration and mosaicking of remote sensing images,'' {\em Optik -
  International Journal for Light and Electron Optics}, vol.~125, no.~2,
  pp.~697 -- 703, 2014.

\bibitem{EuJShP07}
J.~Euzenat and P.~Shvaiko, {\em Ontology Matching}.
\newblock Hiedelberg: Springer-Verlag, 2007.

\bibitem{EuJFeA09}
J.~Euzenat, A.~Ferrara, P.~Hollink, A.~Isaac, C.~Joslyn, V.~Malais\'e,
  C.~Meilicke, A.~Nikolov, J.~Pane, M.~Sabou, F.~Scharffe, P.~Shvaiko,
  V.~Spiliopoulos, H.~Stuckenschmidt, O.~Iv\'ab-Zamazal, V.~Sv\'atek, C.~T. dos
  Santos, G.~Vouros, and S.~Wang, ``Results of the ontology alignment
  evaluation initiative 2009,'' in {\em Proc.\ 4th Int.\ Wshop.\ on Ontology
  Matching (OM-2009), CEUR}, vol.~551, 2009.

\bibitem{JoCPaP09a}
C.~Joslyn, P.~Paulson, and A.~White, ``Measuring the structural preservation of
  semantic hierarchy alignments,'' in {\em Proc.\ 4th Int.\ Wshop.\ on Ontology
  Matching (OM-2009), CEUR}, vol.~551, 2009.

\bibitem{little2009designing}
E.~G. Little and G.~L. Rogova, ``Designing ontologies for higher level
  fusion,'' {\em Information Fusion}, vol.~10, no.~1, pp.~70--82, 2009.

\bibitem{hammer1997extracting}
J.~Hammer, H.~Garcia-Molina, J.~Cho, R.~Aranha, and A.~Crespo, ``Extracting
  semistructured information from the web,'' 1997.

\bibitem{riloff1999learning}
E.~Riloff, R.~Jones, {\em et~al.}, ``Learning dictionaries for information
  extraction by multi-level bootstrapping,'' in {\em AAAI/IAAI}, pp.~474--479,
  1999.

\bibitem{gregory2011domain}
M.~L. Gregory, L.~McGrath, E.~B. Bell, K.~O'Hara, and K.~Domico, ``Domain
  independent knowledge base population from structured and unstructured data
  sources.,'' in {\em FLAIRS Conference}, 2011.

\bibitem{kushmerick2000wrapper}
N.~Kushmerick, ``Wrapper induction: Efficiency and expressiveness,'' {\em
  Artificial Intelligence}, vol.~118, no.~1, pp.~15--68, 2000.

\bibitem{jakobson2004approach}
G.~Jakobson, L.~Lewis, and J.~Buford, ``An approach to integrated cognitive
  fusion,'' in {\em Proceedings of the 7th ISIF International Conference on
  Information Fusion (FUSION2004)}, pp.~1210--1217, 2004.

\bibitem{sarawagi2008information}
S.~Sarawagi, ``Information extraction,'' {\em Foundations and trends in
  databases}, vol.~1, no.~3, pp.~261--377, 2008.

\bibitem{waltz1990multisensor}
E.~Waltz, J.~Llinas, {\em et~al.}, {\em Multisensor data fusion}, vol.~685.
\newblock Artech house Boston, 1990.

\bibitem{DeGroot_2004}
M.~H. DeGroot, {\em Optimal statistical decisions}.
\newblock Wiley, 2004.

\bibitem{hall2004mathematical}
D.~L. Hall and S.~A. McMullen, {\em Mathematical techniques in multisensor data
  fusion}.
\newblock Artech House, 2004.

\bibitem{smith2006approaches}
D.~Smith and S.~Singh, ``Approaches to multisensor data fusion in target
  tracking: A survey,'' {\em Knowledge and Data Engineering, IEEE Transactions
  on}, vol.~18, no.~12, pp.~1696--1710, 2006.

\bibitem{Newman_2013}
A.~J. Newman and G.~E. Mitzel, ``Upstream data fusion: history, technical
  overview, and applications to critical challenges,'' {\em Johns Hopkins APL
  Technical Digest}, vol.~31, no.~3, pp.~215--233, 2013.

\bibitem{alqhtani9multimedia}
S.~M. Alqhtani, S.~Luo, and B.~Regan, ``Multimedia data fusion for event
  detection in twitter by using dempster-shafer evidence theory,'' {\em World
  Academy of Science, Engineering and Technology, International Journal of
  Computer, Electrical, Automation, Control and Information Engineering},
  vol.~9, no.~12, pp.~2234--2238.

\bibitem{crowley1993principles}
J.~L. Crowley, ``Principles and techniques for sensor data fusion,'' in {\em
  Multisensor Fusion for Computer Vision}, pp.~15--36, Springer, 1993.

\bibitem{benferhat2006reasoning}
S.~Benferhat and C.~Sossai, ``Reasoning with multiple-source information in a
  possibilistic logic framework,'' {\em Information Fusion}, vol.~7, no.~1,
  pp.~80--96, 2006.

\bibitem{benferhat2009fusion}
S.~Benferhat and F.~Titouna, ``Fusion and normalization of quantitative
  possibilistic networks,'' {\em Applied Intelligence}, vol.~31, no.~2,
  pp.~135--160, 2009.

\bibitem{Balduzzi_2011}
D.~Balduzzi, ``On the information-theoretic structure of distributed
  measurements,'' {\em EPTCS}, vol.~88, pp.~28--42, 2012.

\bibitem{Thorsen_2006}
S.~Thorsen and M.~Oxley, ``A description of competing fusion systems,'' {\em
  Information Fusion}, vol.~7, pp.~347--360, December 2006.

\bibitem{murphy1999loopy}
K.~P. Murphy, Y.~Weiss, and M.~I. Jordan, ``Loopy belief propagation for
  approximate inference: An empirical study,'' in {\em Proceedings of the
  Fifteenth conference on Uncertainty in artificial intelligence},
  pp.~467--475, Morgan Kaufmann Publishers Inc., 1999.

\bibitem{Lilius_1993}
J.~Lilius, ``Sheaf semantics for {Petri} nets,'' tech. rep., Helsinki
  University of Technology, Digital Systems Laboratory, 1993.

\bibitem{RobinsonQGTopo}
M.~Robinson, ``Imaging geometric graphs using internal measurements,'' {\em J.
  Diff. Eqns.}, vol.~260, pp.~872--896, 2016.

\bibitem{Joslyn_2014}
C.~Joslyn, E.~Hogan, and M.~Robinson, ``Towards a topological framework for
  integrating semantic information sources,'' in {\em Semantic Technologies for
  Intelligence, Defense, and Security (STIDS)}, 2014.

\bibitem{Baclawski_1975}
K.~Bac{\l}awski, ``Whitney numbers of geometric lattices,'' {\em Adv. in
  Math.}, vol.~16, pp.~125--138, 1975.

\bibitem{Baclawski_1977}
K.~Bac{\l}awski, ``Galois connections and the {L}eray spectral sequence,'' {\em
  Adv. Math}, vol.~25, pp.~191--215, 1977.

\bibitem{Shepard_1985}
A.~Shepard, {\em A cellular description of the derived category of a stratified
  space}.
\newblock PhD thesis, Brown University, 1985.

\bibitem{Curry}
J.~Curry, ``Sheaves, cosheaves and applications, {\tt arxiv:1303.3255},'' 2013.

\bibitem{Goldblatt}
R.~Goldblatt, {\em Topoi, the Categorial Analysis of Logic}.
\newblock Dover, 2006.

\bibitem{goguen1992sheaf}
J.~A. Goguen, ``Sheaf semantics for concurrent interacting objects,'' {\em
  Mathematical Structures in Computer Science}, vol.~2, no.~02, pp.~159--191,
  1992.

\bibitem{Spivak_2014}
D.~I. Spivak, {\em Category theory for the sciences}.
\newblock MIT Press, 2014.

\bibitem{Kokar_2004}
M.~Kokar, J.~Tomasik, and J.~Weyman, ``Formalizing classes of information
  fusion systems,'' {\em Information Fusion}, vol.~5, pp.~189--202, 2004.

\bibitem{Kokar_2006}
M.~M. Kokar, K.~Baclawski, and H.~Gao, ``Category theory-based synthesis of a
  higher-level fusion algorithm: An example,'' in {\em Information Fusion, 2006
  9th International Conference on}, pp.~1--8, IEEE, 2006.

\bibitem{SAR_2014}
{Australian Transport Safety Bureau}, ``Considerations on defining the search
  area -- {MH370},'' May 2014.

\bibitem{Stong_1966}
R.~E. Stong, ``Finite topological spaces,'' {\em Trans. Amer. Math. Soc.},
  vol.~123, pp.~325--340, June 1966.

\bibitem{May_2003}
J.~May, ``Finite topological spaces: Notes for {REU},'' 2003.

\bibitem{LH417}
{Flightaware.com}, ``{Lufthansa} {(LH)} \#417 -- 26-{Jul}-2016 -- {KIAD} -
  {FRA/EDDF}, {\tt
  http://flightaware.com/live/flight/dlh417/history/20160726/1925z/kiad/eddf},''
  2016.

\bibitem{scipyopt}
T.~S. community, ``Optimization: scipy.optimize.'' {\tt
  http://docs.scipy.org/doc/scipy/reference/tutorial/optimize.html}.
\newblock accessed Sept 15, 2016.

\bibitem{Hubbard}
J.~H. Hubbard, {\em Teichm\"uller Theory, volume 1.}
\newblock Matrix Editions, 2006.

\bibitem{RobinsonLogic}
M.~Robinson, ``Asynchronous logic circuits and sheaf obstructions,'' {\em
  Electronic Notes in Theoretical Computer Science}, pp.~159--177, 2012.

\bibitem{Purvine_JMM2016}
E.~Purvine, M.~Robinson, and C.~Joslyn, ``Categorification in the real world,''
  in {\em Joint Mathematics Meetings}, (Seattle, WA), January 2016.

\bibitem{Godement_1958}
R.~Godement, {\em Topologie algebrique et th\'eorie des faisceaux}.
\newblock Paris: Herman, 1958.

\bibitem{Bredon}
G.~Bredon, {\em Sheaf theory}.
\newblock Springer, 1997.

\bibitem{TSPbook}
M.~Robinson, {\em Topological Signal Processing}.
\newblock Springer, 2014.

\bibitem{Leray_1950}
J.~Leray, ``L'anneau spectral et l'anneau filtr\'e d'homologie d'un espace
  localement compact et d'une application continue,'' {\em Journal des
  Math\'ematiques Pures et Appliqu\'ees}, vol.~29, pp.~1--139, 1950.

\bibitem{basu2003computing}
S.~Basu, ``Computing the {Betti} numbers of arrangements via spectral
  sequences,'' {\em Journal of Computer and System Sciences}, vol.~67, no.~2,
  pp.~244--262, 2003.

\bibitem{RobinsonSampTA2013}
M.~Robinson, ``The {Nyquist} theorem for cellular sheaves,'' in {\em Sampling
  Theory and Applications}, (Bremen, Germany), 2013.

\bibitem{Ehrenpreis_1956}
L.~Ehrenpreis, ``Sheaves and differential equations,'' {\em Proceedings of the
  American Mathematical Society}, vol.~7, no.~6, pp.~1131--1138, 1956.

\bibitem{DeSilva_2011}
V.~de~Silva, D.~Morozov, and M.~Vejdemo-Johansson, ``Persistent cohomology and
  circular coordinates,'' {\em Discrete \& Computational Geometry}, vol.~45,
  no.~4, pp.~737--759, 2011.

\bibitem{Nanda_2015}
J.~Curry, R.~Ghrist, and V.~Nanda, ``Discrete morse theory for computing
  cellular sheaf cohomology, {\tt arxiv:1312.6454},'' 2015.

\end{thebibliography}

\end{document}